\documentclass[reqno]{amsart}
\usepackage[utf8]{inputenc}
\usepackage{emptypage}
\usepackage{amsopn,amstext,amsbsy,amsmath,amscd,amsthm,amsfonts,systeme,amsfonts,mathtools,amssymb,siunitx}
\usepackage{mathtools}
\usepackage{a4wide}
\usepackage{mathrsfs}
\usepackage{amssymb}
\usepackage{verbatim}
\usepackage{enumerate}
\usepackage{cancel}
\usepackage[shortlabels]{enumitem}
\usepackage{tikz}
\usepackage{mwe}
\usepackage{float}
\usepackage{subfig}
\usepackage{hyperref}
\usepackage{listings}
\usepackage{float}
\usepackage{todonotes}
\usepackage{graphicx}
\usepackage[T1]{fontenc}
\usepackage{amsmath}
\usepackage{latexsym}
\usepackage{dirtytalk}
\usepackage{graphics,epsfig}
\usepackage{color}
\usepackage{a4wide,mathrsfs}
\usepackage{multirow}

\DeclareMathOperator{\codim}{codim}
\usepackage{hyperref}
\usepackage[backend=biber,bibencoding=utf8,firstinits,maxnames=4,style=alphabetic,isbn=false, doi=false, eprint= true, url= false,date=year]{biblatex}
\usepackage[english=american]{csquotes}
\bibliography{literature.bib}

\theoremstyle{plain}
\newtheorem{theorem}        {\bf Theorem}[section]
\newtheorem{corollary}    [theorem]  {\bf Corollary}
\newtheorem{lemma}        [theorem]  {\bf Lemma}
\newtheorem{proposition}  [theorem]  {\bf Proposition} 

\theoremstyle{definition}
\newtheorem{example}      [theorem]  {\bf Example}
\newtheorem{definition}   [theorem]  {\bf Definition}

\theoremstyle{remark}
\newtheorem{remark}      [theorem]  {\bf Remark} 

\numberwithin{equation}{section}

\newcommand{\R}{\mathbb{R}}
\newcommand{\C}{\mathbb{C}}
\newcommand{\N}{\mathbb{N}}
\newcommand{\Z}{\mathbb{Z}}

\newcommand{\Acal}{\mathcal{A}}
\newcommand{\Bcal}{\mathcal{B}}

\newcommand{\Fcal}{\mathcal{F}}

\newcommand{\Lcal}{\mathcal{L}}
\newcommand{\Mcal}{\mathcal{M}}

\newcommand{\Ocal}{\mathcal{O}}

\newcommand{\Tcal}{\mathcal{T}}
\newcommand{\Ucal}{\mathcal{U}}
\newcommand{\Vcal}{\mathcal{V}}
\newcommand{\Wcal}{\mathcal{W}}
\newcommand{\Xcal}{\mathcal{X}}

\newcommand{\dd}{\,\mathrm{d}}

\renewcommand{\phi}{\varphi}
\newcommand{\per}{\mathrm{per}}
\newcommand{\ran}{\mathrm{Ran }}

\DeclareMathOperator{\real}{Re}

\title[Perturbations of embedded eigenvalues]{Perturbations of embedded eigenvalues of asymptotically periodic magnetic Schrödinger operators on a cylinder}
\author{Jonas Jansen, Sara Maad Sasane and Wilhelm Treschow}
\date{\today}
\address{\textit{Jonas Jansen:} Fraunhofer Institute for Algorithms and Scientific Computing SCAI, Schloss Birlinghoven, 53757 Sankt Augustin, Germany}
\email{jonas.jansen@scai.fraunhofer.de}
\address{\textit{Sara Maad Sasane:} Centre for Mathematical Sciences, Lund University, P.O. Box 118, 221 00 Lund, Sweden}
\email{sara.maad\_sasane@math.lth.se}
\address{\textit{Wilhelm Treschow:} Centre for Mathematical Sciences, Lund University, P.O. Box 118, 221 00 Lund, Sweden}
\email{wilhelm.treschow@math.lth.se}

\begin{document}

\begin{abstract}
    We investigate the persistence of embedded eigenvalues for a class of magnetic Laplacians on an infinite cylindrical domain. The magnetic potential is assumed to be $C^2$ and asymptotically periodic along the unbounded direction of the cylinder, with an algebraic decay rate towards a periodic background potential. Under the condition that the embedded eigenvalue of the unperturbed operator lies away from the thresholds of the continuous spectrum, we show that the set of nearby potentials for which the embedded eigenvalue persists forms a smooth manifold of finite and even codimension. The proof employs tools from Floquet theory, exponential dichotomies, and Lyapunov--Schmidt reduction. Additionally, we give an example of a potential which satisfies the assumptions of our main theorem.
\end{abstract}

\maketitle



\section{Introduction}
Embedded eigenvalues are eigenvalues of an operator which also belong to the continuous spectrum. They appear in a variety of different settings, including many applied problems coming from physics, such as in quantum mechanics in which eigenvalues represent the energies of the bound states which the underlying system may attain. In contrast, the continuous spectrum represents the energies attainable by a free particle \cite{reed2005}. Embedded eigenvalues therefore represents a bound state with the same energy as that of a free particle. The behaviour of isolated eigen\-values under perturbations is well known, much due to \cite{kato1995}, in that they persist under small perturbations. Embedded eigenvalues on the other hand typically behave rather differently. As it turns out, they generally cease to exist or turn into resonances even under arbitrarily small perturbations \cite{agmon1989}. It also seems as if finding a complete description of their behavior is much harder than in the case of isolated eigenvalues in that we typically need more detailed assumptions on the underlying problem and more exotic methods. The original ideas of Kato cannot be applied directly due to the fact that they rely on the Riesz projections, which require us to be able to construct a curve contained in the resolvent set, with the only part of the spectrum contained in its interior being the eigenvalue under consideration. This is clearly not possible for embedded eigenvalues, and so we have to resort to other methods.

Perhaps one of the most famous physical examples illustrating the importance of the behavior of embedded eigenvalues under perturbations is the quantum mechanical description of the Helium atom. When studying the Helium atom, a perturbative approach is typically taken in that the repulsion between the electrons is considered a perturbation. The issue is that the unperturbed operator has embedded eigenvalues. If we want to learn anything about the spectrum of the full operator, we thus need to learn how these eigenvalues behave under perturbations \cite{reed2005}.

In this paper, we investigate the problem of persistence of embedded eigenvalues for periodic Schrödinger operators on an infinite cylinder. The study of periodic Schrödinger operators is an area in which a lot of research has been conducted, with great success. Despite this, many problems remain unsolved. The attraction started out with the applications in solid state theory, but has been expanded to areas such as photonic crystals, fluid dynamics, topological insulators and many more.  For an extensive overview of periodic elliptic operators with focus on Schrödinger operators, we refer the reader to \cite{kuchment2016} and the references therein.

More specifically, we consider the magnetic Laplacian on an infinite cylinder with an asymptotically periodic magnetic potential. We define the elliptic operator in question by
\begin{equation*}
    L_{A_0} = - \partial_z^2 - \bigl(\partial_{\phi} + i A_0(z,\phi)\bigr)^2,
\end{equation*}
where \((z,\phi)\in \R\times S^1\) and \(A_0\colon \R\times S^1\to \R\) is a magnetic potential, which is asymptotically periodic in \(z\). We call \(A_0\) asymptotically periodic if there exists a potential $A_{\per} \colon \R \times S^1 \to \R$, which is periodic in $z$, i.e. $A_{\per}(z+p,\phi) = A_{\per}(z,\phi)$ for all \((z,\phi)\in \R\times S^1\) and some \(p>0\), and such that \(|A_0(z,\phi)-A_{\per}(z,\phi)|\to 0\) as \(|z|\to \infty\). Such asymptotically periodic Schrödinger operators can be used to model electrons moving in an infinite crystal, a configuraton assumed by solids consisting of infinitely many identical particles arranged in a periodic array, exhibiting translational invariance with some local impurities accounting for the deviation from the underlying periodic potential \cite{ashcroft1976}.

Furthermore, we assume that $L_{A_0}$ has a simple eigenvalue \(\lambda_0>0\) embedded in the continuous spectrum of \(L_{A_0}\). We point out that it is a priori unclear whether asympotically periodic potentials with this property exist. In Section \ref{sec:ex} we show how to construct infinitely many examples of asymptotically periodic magnetic potentials possessing a simple embedded eigenvalue.

The goal of this paper is to analyze the persistence of the embedded eigenvalue under perturbations of the magnetic potential, which in this case will mean that we replace the potential $A_0$ with another potential $A$ close to $A_0$ in some specified Banach space. To this end, we aim to define a manifold of such perturbations for which the eigenvalue $\lambda_0$ persists, in the sense that there exists an eigenvalue, $\lambda$, of the perturbed operator in a neighborhood of the original eigenvalue, $\lambda_0$. Lastly, we prove that this manifold has a finite and even codimension in the Banach space of perturbations. In order for the codimension to be well-defined, it is required that $\lambda_0$ does not cross any thresholds of the continuous spectrum where the spectrum changes multiplicity. Note that, by Weyl's perturbation theorem, $\sigma_{\mathrm{ess}}(L_A)=\sigma_{\mathrm{ess}}(L_{A_0})=\sigma_{\mathrm{ess}}(L_{A_\per})$ and it is a well-known fact that the spectrum of periodic elliptic operators can be written as a union of, in this case, at most countably many, possibly overlapping, closed intervals, called spectral bands, see e.g. \cite[Sec. 3.6]{plum2011}. The edges of these spectral bands are called spectral edges \cite{kuchment2016}. For the definition of the notion multiplicity in the continuous spectrum, we refer the reader to \cite[Ch. 6.71, Def. 2]{achiezer1993}. Since the generating functions of the spectral bands are linearly independent, precisely at the edges of the spectral bands, the multiplicity may change. These are exactly the points that have to be avoided in order to use Lyapunov–Schmidt reduction in the proof.

\begin{figure}[H]
    \centering
    \includegraphics[width=0.3\linewidth]{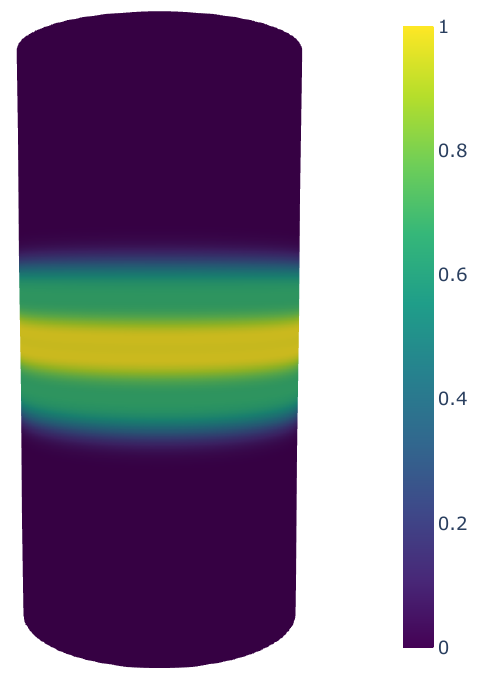}\hspace{2cm}
    \includegraphics[width=0.3\linewidth]{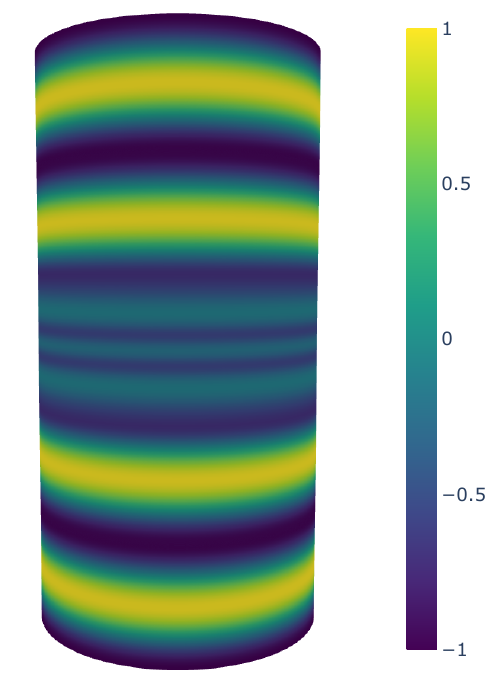}
    \caption{A radially symmetric eigenfunction (left) for an embedded eigenvalue of a magnetic Laplacian with asymptotically periodic magnetic potential (right) on an infinite cylinder.}
    \label{fig:eigenfct}
\end{figure}

Interest in the structure of the set of perturbations that preserve a non-treshold embedded eigenvalue dates back to the work of Agmon, Herbst, and Skibsted \cite{agmon1989}, who conjectured that this set forms something like a differentiable manifold. A rigorous example supporting this conjecture is given in \cite{cruz-sampedro2002}, where the authors consider a Schrödinger operator with an explicit, asymptotically vanishing potential on the real line, first investigated in 1929 by Wigner and von Neumann in \cite{vonneumann1993}. In \cite{cruz-sampedro2002}, the authors show that the set of potentials for which the embedded eigenvalue persists forms a manifold of codimension one.

This line of research was extended to broader classes of operators. In \cite{derks2008}, the authors study the Bilaplacian with an asymptotically vanishing potential on the infinite cylinder and demonstrate that embedded eigenvalues persist under perturbations belonging to a smooth manifold of finite codimension. Like the present manuscript, their method is based on a spatial dynamics formulation of the eigenvalue problem and the theory of exponential dichotomies, enabling the application of Lyapunov–Schmidt reduction to identify the manifold of admissible perturbations. For the magnetic Schrödinger operator considered here, this method was previously applied in the case of asymptotically constant magnetic potentials in \cite{laptev2017}.

Following the example in \cite{cruz-sampedro2002}, the Schrödinger operator on the real line has also been studied using this approach: for asymptotically constant potentials in \cite{maadsasane2023} and for asymptotically periodic operators in \cite{maadsasane2024}. In the latter case, Floquet theory is used to derive the exponential dichotomies.

An alternative approach to studying the persistence of embedded eigenvalues for self-adjoint operators is presented in \cite{agmon2011}. There, the authors consider bounded perturbations and, under the assumption of finite multiplicity of the continuous spectrum near the embedded eigenvalue, construct a finite-codimension manifold of perturbations for which the eigenvalue persists. This result relies on the limiting absorption principle and a detailed perturbation theory of the spectral density. Although this method appears extendable to relatively bounded perturbations, only the spatial dynamics method, as used in this manuscript, has so far been extended to operators with continuous spectra of infinite multiplicity. See \cite{derks2011} for such a study in the case of the planar Bilaplacian with a compactly supported, radially symmetric potential under general compactly supported perturbations.

As noted, this manuscript is an extension of a study of the magnetic Schrödinger operator on the infinite cylinder with asymptically constant potential \cite{laptev2017}. As in \cite{derks2008} and \cite{laptev2017}, the eigenvalue problem is reformulated via spatial dynamics into an ill-posed dynamical system, where the unbounded spatial direction is treated as time. However, unlike in previous studies, the system in our case is not asymptotically autonomous, introducing significant technical complications. We apply Floquet theory to analyze the spectrum of the evolution operator and establish exponential dichotomies. This, in turn, enables the use of Lyapunov–Schmidt reduction to determine the structure of the manifold of admissible perturbations.

\begin{remark}
    In Section \ref{sec:ex} we provide a method that can be used to find an infinite family of examples of an operator with the required properties, as well as one explicit example.
\end{remark}

\section{The Main Result}

For $\beta >1$, we define the space of admissible perturbations 
\begin{equation*}
    \Xcal_{\beta} = \{B\in C^2(\R\times S^1)\,:\, \|B\|_{\beta} < \infty\},
\end{equation*}
where 
\begin{equation*}
    \|B\|_{\beta} = \sup\limits_{(z,\phi)\in \R\times S^1} |D^2B(z,\phi)| + \sup\limits_{(z,\phi)\in \R\times S^1} \Bigl(|B(z,\phi)| + \Bigl|\frac{\partial B}{\partial\phi}(z,\phi)\Bigr|\Bigr)(1+|z|)^{\beta}.
\end{equation*}

We make the following assumptions:

\begin{enumerate}[label=(A.\arabic*)]
    \item\label{it:A.1} \(A_0\in C^2(\R\times S^1)\) and there is a $p$-periodic magnetic potential \(A_{\per}\in C^2(\R\times S^1)\) such that $A_0-A_{\per}\in \Xcal_{\beta}$ for some $\beta >1$;
    \item\label{it:A.2} there exists $z_0\in \R$ such that \(\int_{S^1} \left(A_0(z_0,\phi)-A_{\per}(z_0,\phi)\right)\dd \phi \neq 0\);
    \item\label{it:A.3} \(\lambda_0>0\) is a simple eigenvalue of \(L_{A_0}\), which is (without loss of generality) embedded in the continuous spectrum of \(L_{A_0}\);
    \item\label{it:A.4} \(\lambda_0\) is not situated on a threshold of the spectrum of $L_{A_0}$, i.e., a point where the continuous spectrum changes multiplicity. In particular, \(\lambda_0\notin\sigma_{\mathrm{per}}(L_{A_{\per}})\).
\end{enumerate}
Here, we denote by \(\sigma_{\mathrm{per}}(L_{A_{\per}})\) the spectrum of the operator \(L_{A_{\per}}\) on \(L^2_{\per}([0,p]\times S^1)\), where \(L^2_{\per}([0,p]\times S^1)\) denotes the space of functions on \(\R\times S^1\), which are $p$-periodic with respect to the \(z\)-variable and locally square integrable. We remark that the second part of \ref{it:A.3} is not a restriction, as the corresponding result for the case where $\lambda_0$ is an isolated eigenvalue holds by the extensive previous work on perturbation theory pertaining to isolated eigenvalues, see e.g. \cite{kato1995}. Our method can also handle this case, and then the proof is shorter than the general case. Hence, we focus on the more difficult case where the result is not well-known. By a threshold of the spectrum we mean an edge of one of the spectral bands.

We want to analyze the persistence of the embedded eigenvalue \(\lambda_0\) under small perturbations of the magnetic potential $A_0$. Therefore, we define the set
\begin{equation*}
    \Mcal_{\delta} = \{A\in C^2(\R\times S^1)\, : \, A-A_0\in \Xcal_{\beta} \text{ and there exists }\lambda \in (\lambda_0-\delta,\lambda_0+\delta) \text{ such that } \lambda \in \sigma_{p}(L_{A})\}
\end{equation*}
for \(\delta >0\) sufficiently small. Here, $\sigma_p$ denotes the point spectrum. Note further that $\Mcal_{\delta}$ is non-empty since \(A_0\in \Mcal_{\delta}\). The aim of this paper is to analyze the structure of the set $\Mcal_{\delta}$. We will show that, locally around $A_0$, $M_{\delta}$ defines a manifold of perturbations for which the eigenvalue persists and we determine its codimension, leading to the following result.

\begin{theorem}\label{thm:main}
    Suppose that \ref{it:A.1}--\ref{it:A.4} hold. Then there exists \(\delta >0\), \(m \in \N\) and a neighborhood \(\Ocal\) of \(A_0\) in \(A_0 + \Xcal_{\beta}\) such that \(M_{\delta} \cap \Ocal\) is a smooth manifold of codimension \(2m\).
\end{theorem}

\begin{remark}
Note that the natural number $m$, to be properly defined later, is directly related to the multiplicity of the continuous spectrum at $\lambda_0$. In particular, this implies that the result is consistent with the known theory for isolated eigenvalues, in which case the multiplicity of the continuous spectrum at $\lambda_0$ is $0$, giving that $m=0$. This also highlights the importance of assumption \ref{it:A.4}: if the eigenvalue moves so that the multiplicity of the continuous spectrum changes from $\lambda_0$ to $\lambda$, the methods do not work and thus the result does not hold.
\end{remark}

The methods we use combine those in \cite{laptev2017} with more general methods from spatial dynamics and Floquet theory. The main idea for proving Theorem \ref{thm:main} is to find a sufficiently regular function depending on $\lambda$ and $A$, defined in a neighborhood of $(\lambda_0,A_0)$, which evaluates to zero if and only if $\lambda$ is an eigenvalue of the operator with potential $A$, in order to be able to apply the implicit funtion theorem to obtain the dependence of $\lambda$ being an eigenvalue of $L_A$ with magnetic potential $A$. The structure of the paper is as follows:

In Section \ref{sec:ODEformulation}, we introduce the spatial dynamics formulation of the eigenvalue problem. We then obtain exponential dichotomies for a shifted version of this ill-posed dynamical system, so that we obtain evolution operators acting on subspaces of bounded functions and functions decaying exponentially as $z\to +\infty$ and $z\to -\infty$, respectively.

In Section \ref{sec:exponentialdecay}, we use these evolution operators to prove that any eigenfunction must be exponentially decaying.

Section \ref{sec:Lyapunov} contains the proof of the main theorem. This entails the construction of the sufficiently regular function $\iota$ and the application of Lyapunov--Schmidt reduction to obtain $\lambda$ as a function of the potential $A$.

Finally, in Section \ref{sec:ex} we provide an infinite family of examples of magnetic Schrödinger operators on the infinite cylinder possessing simple embedded eigenvalues.

\section{The ODE Formulation}\label{sec:ODEformulation}

We now study the eigenvalue equation
\begin{equation}\label{eq:eigenvalue-problem}
    L_{A} u = \lambda u, \quad A \in A_0 + \Xcal_{\beta}
\end{equation}
in more detail. In order to do this, we rewrite the system using spatial dynamics. To simplify notation, we denote by \(\partial \colon H^j(S^1) \to H^{j-1}(S^1)\) the operator of differentiation in the \(\phi\)-direction. Then we may rewrite the eigenvalue problem \eqref{eq:eigenvalue-problem} as
\begin{equation}\label{eq:eigenvalue-spatial-dynamics}
    \Ucal' = \Bcal(z;\lambda,A) \Ucal,
\end{equation}
where the prime means differentiation with respect to \(z\), \(\Ucal = (u,u')^T\) and
\begin{equation*}
    \Bcal(z;\lambda,A) = \begin{pmatrix}
        0 & 1 \\
        -(\partial + iA)^2 - \lambda & 0
    \end{pmatrix}.
\end{equation*}
Note that the magnetic potential \(A\) acts on functions in \(C(\R;H^k(S^1))\) via \(Au(z) = A(z,\cdot)u(z)\). The matrix operator \(\Bcal(z;\lambda,A)\) acts on the space $X=H^1(S^1)\times L^2(S^1)$ with domain \(D(\Bcal)=H^2(S^1)\times H^1(S^1)\). The eigenvalue equation and the spatial dynamics formulation are equivalent descriptions.

\begin{lemma}\label{lem:equivalence-spatial-dynamics}
    \(u\in H^2_{\mathrm{loc}}(\R\times S^1)\) solves the eigenvalue problem \eqref{eq:eigenvalue-problem} if and only if the spatial dynamics formulation \eqref{eq:eigenvalue-spatial-dynamics} has a solution \(\Ucal\in C(\R;D(\Bcal)) \cap C^1(\R;X)\).
\end{lemma}
A proof of this result can be found in \cite[Lem. ~1]{laptev2017}.

\subsection{The system at infinity}

The next goal is to show that the eigenfunctions of \(L_A\) decay exponentially as \(|z|\to \infty\). To obtain this result, we must first analyze the asymptotic behaviour of solutions to \eqref{eq:eigenvalue-spatial-dynamics} and construct evolution operators corresponding to the spatial dynamics formulation, using exponential dichotomies as a tool.

As per assumption \ref{it:A.1}, the magnetic potential \(A_0\) and its perturbations are asymptotically periodic, we may study the system at infinity given by
\begin{equation}\label{eq:sys-at-inf}
    \Ucal' = \Bcal(z;\lambda,A_{\per}) \Ucal.
\end{equation}
We define the operator \(\Tcal\colon L^2(\R;X)\to L^2(\R;X)\) by
\begin{equation}\label{eq:system-at-infinity}
    \Tcal\Ucal = \dfrac{d}{dz}\Ucal - \mathcal{B}(z;\lambda,A_{\per})\Ucal = \dfrac{d}{dz}\Ucal - \bigl(\Acal + \Lambda_{\per}(z;\lambda)\bigr)\Ucal,
\end{equation}
where
\begin{equation*}
    \Acal = \begin{pmatrix}
        0 & 1 \\ - \partial^2 & 0
    \end{pmatrix} \quad \text{and} \quad  \Lambda_{\per}(z;\lambda) = \begin{pmatrix}
        0 & 0 \\ A_{\per}(z,\cdot)^2 - i(\partial A_{\per}) (z,\cdot) - 2i A_{\per}(z,\cdot)\partial - \lambda & 0
    \end{pmatrix}
\end{equation*}
with domain \(D(\Tcal) = H^1(\R;X)\cap L^2(\R;D(\Bcal))\). Note that \(\Acal\) is closed, densely defined and, by the Rellich--Kondrachev theorem, has a compact resolvent.

As \(A_{\per}\) is periodic, we associate to \(\Tcal\) the Floquet operator \(\Tcal_{\per}:L_{\per}^2([0,p],X)\to L_{\per}^2([0,p],X)\), which is defined by
\begin{equation}
    \Tcal_{\per}\Ucal = \dfrac{d}{dz} \Ucal - (\Acal  + \Lambda_{\per}(z;\lambda))\Ucal, \quad D(\Tcal_{\per}) = H^1_{\per}([0,p];X) \cap L^2_{\per}([0,p];D(\Bcal)).
\end{equation}

The eigenvalues \(\alpha\in \sigma_{p}(\Tcal_{\per})\) are called \emph{Floquet exponents} of \(\Tcal\). In particular, \(\Tcal_{\per} \Ucal_{\alpha} = \alpha\Ucal_{\alpha}\) implies that \(\Ucal(z) = e^{-\alpha z} \Ucal_{\alpha}(z)\) is a solution to \eqref{eq:sys-at-inf}. We want to prove that the set of Floquet exponents is discrete. To do this, we first need some preliminary results.
\begin{lemma}\label{lem:per-ref-op-spec}
    The operator $\Tilde{\Tcal}_{\per}\coloneqq \dfrac{d}{dz}-\Acal$ acting on \(L_{\per}^2([0,p],X)\) has compact resolvent.
\end{lemma}
\begin{proof}
The eigenvalue equation of this operator is
\begin{equation*}
    \left(\dfrac{d}{dz} -\Acal- \alpha\right)\Ucal=0.
\end{equation*}
By periodicity in the $z$-variable, we can expand $\Ucal$ in a Fourier series as 
\begin{equation*}
    \Ucal=\sum_{k\in\Z} \Ucal_ke^{-\frac{2\pi ki}{p}z}.
\end{equation*}
Then, for $k\in\Z$,
\begin{equation*}
    \left(-\dfrac{2\pi ki}{p} - \Acal - \alpha\right)\Ucal_k=0.
\end{equation*}
Rewriting yields
\begin{equation*}
    \Acal\Ucal_k=-\left(\dfrac{2\pi ki}{p}+\alpha\right)\Ucal_k.
\end{equation*}
Thus, $-\left(\dfrac{2\pi ki}{p}+\alpha\right)$ must be an eigenvalue of $\Acal$ and $\Ucal_k$ an eigenvector, for $k\in\Z$. Hence, the eigenvalues of the operator $\Tilde{\Tcal}_{\per}$ are
\begin{equation*}
    \alpha_{nk} = -\mu_n - \dfrac{2\pi ki}{p}, \quad n\in\N, \hspace{2mm} k\in\Z,
\end{equation*}
where $\mu_n$ are eigenvalues of $\Acal$.
Thus, the operator $\Tilde{\Tcal}_{\per}$ has discrete spectrum and hence the resolvent set is non-empty. This means that, for $\alpha$ in the resolvent set, $(\Tilde{\Tcal}_{\per} - \alpha)^{-1}$ is a bounded operator from $L_{\per}^2([0,p],X)$ to $D(\Tilde{\Tcal}_{\per})=H_{\per}^1([0,p],X)\cap L_{\per}^2([0,p],D(\Acal))$, and hence compact from $L_{\per}^2([0,p],X)$ to itself by the Aubin--Lions lemma.
\end{proof}
The following lemma follows the ideas of \cite{mielke1996}, with slight changes due to the different setting.
\begin{lemma}\label{lem:cond-comp-res}
    Exactly one of the following holds:
    \begin{enumerate}[(i)]
        \item For all $\alpha\in\C$ there exists a $\Ucal\in D(\Tcal_{\per})\setminus\{0\}$ with $\Tcal_{\per}\Ucal=\alpha\Ucal$.
        \item The spectrum of $\Tcal_{\per}$ is discrete and the resolvent is compact.
    \end{enumerate}
\end{lemma}
\begin{proof}
    Assume that $(i)$ does not hold. Then there exists an $\alpha\in\C$ such that $\Tcal_{\per}\Ucal=\alpha\Ucal$ has only the trivial solution. Rewriting, the eigenvalue equation becomes
    $$\Tcal_{\per}\Ucal = \Tilde{\Tcal}_{\per}\Ucal - \Lambda_{\per}(\lambda)\Ucal= (\Tilde{\Tcal}_{\per}-\Tilde{\alpha})\Ucal + \Tilde{\alpha}\Ucal - \Lambda_{\per}(\lambda)\Ucal =\alpha\Ucal.$$
    Applying the resolvent $(\Tilde{\Tcal}_{\per}-\Tilde{\alpha})^{-1}$, for some $\Tilde{\alpha}$ in the resolvent set of $\Tilde{\Tcal}_{\per}$, and rewriting yields
    $$\Ucal - (\Tilde{\Tcal}_{\per}-\Tilde{\alpha})^{-1}\Lambda_{\per}(\lambda)\Ucal = (\alpha-\Tilde{\alpha})(\Tilde{\Tcal}_{\per}-\Tilde{\alpha})^{-1}\Ucal.$$
    Thus, the equation can be rewritten as
    $$M(\alpha;\lambda)\Ucal =\Ucal,$$
    where $M(\alpha;\lambda)\coloneqq (\Tilde{\Tcal}_{\per}-\Tilde{\alpha})^{-1}(\Lambda_{\per}(\lambda) + \alpha - \Tilde{\alpha})$, has only the trivial solution. The operator $M(\alpha;\lambda)$ is clearly compact, as the composition of a compact and a bounded operator. By Fredholm theory, it follows that $1-M(\alpha;\lambda)$ is invertible. Since
    \begin{equation*}
        \begin{aligned}
            (\Tilde{\Tcal}_{\per}-\Tilde{\alpha})(1-M(\alpha;\lambda))&=(\Tilde{\Tcal}_{\per}-\Tilde{\alpha})(1-(\Tilde{\Tcal}_{\per}-\Tilde{\alpha})^{-1}(\Lambda_{\per}(\lambda) + \alpha - \Tilde{\alpha}) \\
            &= \Tilde{\Tcal}_{\per}-\Tilde{\alpha} - \Lambda_{\per}(\lambda) - \alpha + \Tilde{\alpha}\\
            &=\underbrace{\Tilde{\Tcal}_{\per}-\Lambda_{\per}(\lambda)}_{=\Tcal_{\per}} - \alpha
        \end{aligned}
    \end{equation*}
    and $(\Tilde{\Tcal}_{\per}-\Tilde{\alpha})^{-1}$ is compact, it follows that
    $$(\Tcal_{\per}-\alpha)^{-1} = (1-M(\alpha;\lambda))^{-1}(\Tilde{\Tcal}_{\per}-\Tilde{\alpha})^{-1}$$
    is the compact resolvent of $\Tcal_{\per}$. In particular $(ii)$ holds.

    Clearly, $(i)$ and $(ii)$ can not be true simultaneously. Additionally, if the spectrum is not discrete, then if $(i)$ does not hold, it follows by the above that $(ii)$ holds and we arrive at a contradiction and so $(i)$ must hold. Then the resolvent is not defined for any $\alpha\in\C$ and so no compact resolvents exist.
\end{proof}
Now we can prove that the set of Floquet exponents is discrete.

\begin{lemma}\label{lem:discreteness-of-spectrum}
     Assume \ref{it:A.4}. Then the spectrum \(\sigma(\Tcal_{\per})\) of \(\Tcal_{\per}\) consists only of Floquet exponents and is discrete. Furthermore, the operator $\Tcal_{\per}$ has compact resolvent. 
\end{lemma}

\begin{proof}
    As can be noted in the proof of Lemma \ref{lem:cond-comp-res}, it is enough to prove that there is an $\alpha\in\C$ such that $\Tcal_{\per}\Ucal=\alpha \Ucal$ only has the trivial solution or, in other words, that the operator $\Tcal_{\per}-\alpha I$ is invertible for some $\alpha\in\C$. Consider the equation
    \begin{equation*}
        (\Tcal_{\per}-\alpha I)\Ucal=0,
    \end{equation*}
    or more explicitly
    \begin{equation*}    
    \dfrac{d}{dz}\Ucal-\begin{pmatrix}
        \alpha & 1 \\
        -\big(\partial + iA_{\per}(z,\cdot)\big)^2-\lambda & \alpha
    \end{pmatrix}\Ucal=0.
    \end{equation*}
For this to hold, we must have that
\begin{equation*}
    \begin{cases}
    \partial_z u_1=\alpha u_1 + u_2 \\
    \partial_z u_2 = -\big(\partial_{\phi} + iA_{\per}(z,\phi)\big)^2u_1 -\lambda u_1+ \alpha u_2.
    \end{cases}
\end{equation*}
In particular, from the first equation we can see that it must hold that $u_2=\partial_z u_1-\alpha u_1$. Substituting this into the second equation gives us that 
\begin{equation}
\underbrace{-\partial_z^2u-\big(\partial_{\varphi}+A_{\per}(z,\varphi)\big)^2u-\lambda u}_{=(L_{A_{\per}}-\lambda I)u}+2\alpha u_z-\alpha^2u = 0.
\label{eq:alphaeq}
\end{equation}
Note that the operator $L_{A_{\per}}$ is acting on the Hilbert space $L_{\per}^2([0,p]\times S^1)$ with periodic boundary conditions, and therefore is self-adjoint and has a compact resolvent for every $\lambda\notin\sigma_{per}(L_{A_{\per}})$. Moreover, since we assumed \ref{it:A.4}, it follows that  $L_{A_{\per}}-\lambda$ is invertible on $L_{\per}^2([0,p]\times S^1)$ and that the inverse is a compact operator. Thus, we may equivalently write \eqref{eq:alphaeq} as 
\begin{equation*}
\big(I+(L_{A_{\per}}-\lambda)^{-1}(2\alpha\partial_z - \alpha^2I)\big)u=0.
    \label{eq:alphaeq2}
\end{equation*}
We define the operator-valued function 
\begin{equation*}
    \C \ni\alpha \mapsto K(\alpha)\in \mathcal{L}\big(L_{\per}^2([0,p]\times S^1)\big),
\end{equation*}
where
\begin{equation*}
    K(\alpha)=-(L_{A_{\per}}-\lambda)^{-1}(2\alpha\partial_z - \alpha^2I)
\end{equation*}
is analytic on all of $\C$. Moreover, keeping in mind that the operator $(2\alpha\partial_z-\alpha^2I)$ is relatively bounded with respect to $L_{\per}$, we clearly have that $K(\alpha)$ is compact for each $\alpha\in \C$. Thus, the existence of $(\Tcal_{\per}-\alpha I)^{-1}$ for some $\alpha\in\C$ is equivalent to the existence of $\big(I-K(\alpha)\big)^{-1}$.
According to the Analytic Fredholm theorem \cite[Thm.~8.92]{renardy2004}, we have that either
\begin{enumerate}
    \item $\big(I-K(\alpha)\big)^{-1}$ exists for no $\alpha\in \C$, or
    \item $\big(I-K(\alpha)\big)^{-1}$ exists for every $\alpha\in \C\setminus S$, where $S$ is a discrete set in $\C$. In this case the function $\alpha\mapsto\big(I-K(\alpha)\big)^{-1}$ is analytic on $\C\setminus S$, and if $\alpha\in S$, then $K(\alpha)u=u$ has a finite-dimensional family of solutions.
\end{enumerate}
It is clear that we end up in the second case, as $\big(I-K(\alpha)\big)^{-1}$ exists for $\alpha=0$. This concludes the proof.
\end{proof}
\begin{lemma}
    Assume \ref{it:A.4}. Then \(\alpha = 0\) is not a Floquet exponent of \(\Tcal\).
\end{lemma}

\begin{proof}
    To show this, we first prove that $\alpha=0$ is a Floquet exponent if and only if $\lambda\in\sigma_{\per}(L_{A_{\per}})$.

    If $\alpha=0$ is a Floquet exponent, then there exists a non-trivial $\Ucal_0$ such that $\Tcal_{\per}\Ucal_0=0$. This means that, setting $\Ucal_0=(u_0,v_0)^T$, we have
    \begin{equation*}
        \begin{aligned}
            \begin{cases}
                u_0'=v_0 \\
                v_0' = -\big(\partial + iA(z,\cdot)\big)^2u_0 - \lambda u_0
            \end{cases} &\iff \begin{cases}
                v_0=u_0' \\
                u_0'' = -\big(\partial + iA(z,\cdot)\big)^2u_0 - \lambda u_0
            \end{cases} \\
            &\iff \begin{cases}
                v_0=u_0' \\
                L_{A_{\per}}u_0 = \lambda u_0.
            \end{cases}
        \end{aligned}
    \end{equation*}
    The other direction is trivial. Now, since we assumed in \ref{it:A.4} that $\lambda_0\notin\sigma_{\per}(L_{A_{\per}})$, the result follows by the discreteness of $\sigma_{\per}(L_{A_{\per}})$ for $\lambda$ close to $\lambda_0$.
\end{proof}

\begin{lemma}\label{lem:existance-of-purely-imaginary-floquet-exp}
    If $\lambda\in\sigma(L_{A_{\per}})$, then there exists a purely imaginary Floquet exponent $\alpha\in i\R$ of $\Tcal$.
\end{lemma}
\begin{proof}
It is well known that since $L_{A_{\per}}=L_{A_{\per}}(D)$ is a periodic elliptic differential operator, then the spectrum of $L_{A_{\per}}$ is a union of closed intervals, i.e.,
\[\sigma(L_{A_{\per}})=\bigcup_{k\in\N} I_k,\]
where $I_k=\{\lambda_k(\theta) \,: \, \theta\in [-\pi/p,\pi/p]\}$, and $\lambda_k(\theta)$ are the eigenvalues corresponding to the eigenvalue problem on the fundamental domain of periodicity with semi-periodic boundary conditions
\begin{equation*}
    \begin{aligned}
        &L_{A_{\per}}(D)u=\lambda u, \\
        &u(z+p)=e^{i\theta p}u(z).
    \end{aligned}
\end{equation*}
Or, equivalently they can be seen as eigenvalues corresponding to a transformed eigenvalue problem on the fundamental domain of periodicity with periodic boundary conditions
\begin{equation*}
    \begin{aligned}
        &L_{A_{\per}}(D+i\theta)u=\lambda u, \\
        &u(z+p)=u(z).
    \end{aligned}
\end{equation*}
Clearly, if $\lambda\in\sigma(L_{A_{\per}})$, there must exist a $k_0\in\N$ and $\theta_0\in [-\pi/p,\pi/p]$ such that $\lambda=\lambda_{k_0}(\theta_0)$. Thus, it follows by the equivalence of formulations that $\alpha=i\theta_0$ is a Floquet exponent. For more details, see \cite[Sec. 3.6]{plum2011}.
\end{proof}

Next, as $\sigma(\Tcal_{\per})$ is discrete by Lemma \ref{lem:discreteness-of-spectrum}, we may define a continuous family of spectral projection operators $P(z)$, $z\in \R$, for $\Tcal_{\per}$ as in \cite[Sec. 2.4]{mielke1996}. Note that if $\alpha \in \sigma(\Tcal_{\per})\cap i\R$, then also $\alpha+2\tfrac{\pi}{p}i\in \sigma(\Tcal_{\per})$ and the spectral projection $P(0)$ onto $\sigma(\Tcal_{\per})\cap i\R$ may be written as
\begin{equation}
    P(0) \Ucal = \sum_{k=1}^{N} \langle \Ucal, \psi_k \rangle \phi_k,
    \label{eq:spect-proj}
\end{equation}
where $\phi_k$ are eigenfunctions of $\Tcal_{\per}$ and $\psi_k$ are eigenfunctions of $\Tcal_{\per}^*$. We are now able to define
\begin{equation}
    X^c = \mathrm{Ran}\,P(0) = \operatorname{span}\{\phi_1,\ldots,\phi_N\}
    \label{eq:center-man-def}
\end{equation}
to be the central part of $\Tcal_{\per}$.

\begin{proposition}
    Assume \ref{it:A.4}. Then \(\dim X^c\) is positive and even.
\end{proposition}

\begin{proof}
    That the dimension is positive follows directly from Lemma \ref{lem:existance-of-purely-imaginary-floquet-exp} since $\lambda_0\in\sigma_{\mathrm{ess}}(L_{A_0})=\sigma_{\mathrm{ess}}(L_{A_{\per}})\subset \sigma(L_{A_{\per}})$, and the finiteness follows from \eqref{eq:spect-proj}--\eqref{eq:center-man-def}.

    We can use the projection to decompose our eigenvalue problem into two parts. One finite-dimensional, taking projections onto $X^c$, and one infinite-dimensional part (for more on this, see \cite[Thm.~6.17]{kato1995}).

Now, let $R:X\to X$ be the involution defined by 
\begin{equation*}    
R:\Ucal\mapsto \begin{pmatrix}
    r & 0 \\
    0 & -r
\end{pmatrix}\Ucal,
\end{equation*}
where $(rf)(z)= f(-z)$. It holds that $R^2 = I$, and if we define
\begin{equation*}
    \ell_{A_{\per}}(z) = -(\partial+iA_{\per}(z,\cdot))^2-\lambda I,
\end{equation*} we have that
\begin{equation*}
    \begin{aligned}
        \Bcal(z;\lambda,A_{\per})R = \begin{pmatrix}
            0 & 1 \\
            \ell_{A_{\per}}(z) & 0
        \end{pmatrix}\begin{pmatrix}
            r & 0 \\
            0 & -r 
        \end{pmatrix} = \begin{pmatrix}
            0 & -r \\
            \ell_{A_{\per}}(z)r & 0
        \end{pmatrix}.
    \end{aligned}
\end{equation*}
Furthermore,
\begin{equation*}
    \begin{aligned}
        R\Bcal(-z;\lambda,A_{\per}) &= \begin{pmatrix}
            r & 0 \\
            0 & -r
        \end{pmatrix}\begin{pmatrix}
            0 & 1 \\
            \ell_{A_{\per}}(-z) & 0 
        \end{pmatrix} = \begin{pmatrix}
            0 & r \\
            -r\ell_{A_{\per}}(-z) & 0
        \end{pmatrix} \\
        &= \begin{pmatrix}
            0 & r \\
            -\ell_{A_{\per}}(z)r & 0
        \end{pmatrix} = -\Bcal(z;\lambda,A_{\per})R.
    \end{aligned}
\end{equation*}

Now we have, following the ideas of Remark 4.2 in \cite{mielke1996}, that $\Phi^c(z,z_0)$ is the fundamental matrix of the reduced linear problem. By the \say{reversibility} above, we know that $\Phi^c(-z,-z_0) = R^c\Phi^c(z,z_0)R^c$, in addition to the standard properties $\Phi^c(z_0,z)^{-1} = \Phi^c(z,z_0) = \Phi^c(z+p,z_0+p)$. As a consequence, the monodromy matrix $M = \Phi^c(p,0)$ satisfies $M = \Phi^c(0,p)=\Phi^c(-p,0)=R^cMR^c$. Since $\det \Phi^c(z,z_0)>0$ and $\det R^c = 1,$ we find that $\det M = 1,$ so the multiplicity of the eigenvalue $-1$ must be even and all the other eigenvalues (Floquet multipliers) are on the unit circle and appear either in pairs or are $+1$. However, since $\alpha=0$ is not a Floquet exponent, $+1$ is not an eigenvalue of the monodromy matrix, and the result follows. 
\end{proof}

\subsection{Exponential dichotomies}
Next, we introduce the concept of exponential dichotomies, which are essential to our approach for defining evolution operators for the spatial dynamics formulation.
\begin{definition}[Exponential dichotomy]
    Let $J=\R_-$, $\R_+$ or $\R$. An ODE system 
    \begin{equation}\label{eq:system-exp-dich}
        \Ucal' = \Bcal(z)\Ucal,
    \end{equation}
    where $\Ucal(z)\in X$ and $\Bcal(z): X\to X$, is said to possess an exponential dichotomy on $J$ if there exists a family of projections $P(z) \in \Lcal(X)$, $z\in J$, such that the projections satisfy $P(\cdot)\Ucal \in C(J; X)$ for any $\Ucal\in X$, and there exist constants $K$ and $\kappa^s < 0 < \kappa^u$  with the following properties:
    \begin{enumerate}[(i)]
        \item For any $z_0\in J$ and $\Ucal\in X$ there exists a unique solution $\Phi^s(z, z_0)\Ucal$ of \eqref{eq:system-exp-dich} defined for $z \geq z_0$, $z, z_0 \in J$, such that $\Phi^s(z_0, z_0)\Ucal = P(z_0)U$ and
        \begin{equation*}
            \lVert \Phi^s(z,z_0)\Ucal\lVert_X \leq Ke^{\kappa^s(z-z_0)}\lVert\Ucal\rVert_X,
        \end{equation*}
        for every $z \geq z_0$, $z, z_0 \in J$.
        \item For any $z_0\in J$ and $\Ucal\in X$, there exists a unique solution $\Phi^u(z,z_0)\Ucal$ defined for $z\leq z_0$, $z,z_0\in J$, such that $\Phi^u(z_0,z_0)\Ucal = (I-P(z_0))\Ucal$ and
        \begin{equation*}
            \lVert \Phi^u(z,z_0)\Ucal\rVert_X\leq Ke^{\kappa^u(z-z_0)}\lVert \Ucal \rVert_x,
        \end{equation*}
        for every $z\leq z_0$, $z,z_0\in J$.
        \item The solutions $\Phi^s(z,z_0)\Ucal$ and $\Phi^u(z,z_0)\Ucal$ satisfy 
        \begin{equation*}
            \begin{aligned}
                \Phi^s(z,z_0)\Ucal &\in \ran P(z) \quad \text{for all } z\geq z_0, \; z,z_0\in J, \\
                \Phi^u(z,z_0)\Ucal &\in \ker P(z) \quad \text{for all } z\leq z_0, \; z,z_0\in J.
            \end{aligned}
        \end{equation*}
    \end{enumerate}
\end{definition}

As the center space $X^c$ was proven to be non-empty, the system at infinity clearly does not possess an exponential dichotomy.
Therefore, we must introduce a shift, $\pm\eta$, such that, for a new shifted system, none of the Floquet exponents are purely imaginary. This is true by the discreteness of the Floquet exponents.
\begin{equation}\label{eq:exponential-shift-periodic}
    \Vcal_{\pm}' = \big(\Bcal(z;\lambda, A_{\per}) \pm \eta I \big)\Vcal_{\pm}.
\end{equation}

\begin{lemma}\label{lem:exponential-shift-periodic}
    Assume \ref{it:A.4}. Then there exists \(\delta > 0\) such that for any \(\eta\in (0,\delta)\) the system \eqref{eq:exponential-shift-periodic} possesses an exponential dichotomy on \(\R\).
\end{lemma}

\begin{proof}
    By Theorem 2.3 in \cite{mielke1996}, it follows that $\Tcal_{\pm\eta} \coloneqq \Tcal -(\pm \eta) I$ is invertible for all $\eta \in (0,\delta(\lambda))$, with 
    \begin{equation*}
        \delta(\lambda)=\min_{\substack{\alpha\in\sigma(\Tcal_{\per}(\lambda)) \\ \real(\alpha)\neq 0}}\lvert \real(\alpha)\rvert >0,
    \end{equation*}
    which exists and is well-defined by the discreteness of the spectrum.
    The rest follows from the proof of Lemma 5.5 in \cite{sandstede2001} and related results. The arguments apply verbatim with $\Tcal\coloneqq\Tcal_{\pm\eta}$ and $\Tcal_{\text{ref}}\coloneqq \dfrac{d}{dz} - \left(\Acal\pm \eta\right)$.
\end{proof}

Finally, we can now show that the full system, once it is shifted, also possesses an exponential dichotomy. Again, using the change of variables \(\Vcal_{\pm} = e^{\pm \eta z}\Ucal\), we obtain the shifted system

\begin{equation}\label{eq:exponential-shift}
    \Vcal_{\pm}' = \big(\Bcal(z;\lambda, A) \pm \eta I \big)\Vcal_{\pm}.
\end{equation}

Note that we may rewrite this as

\begin{equation}\label{eq:perturb-shift}
    \Vcal_{\pm}'=\big(\Bcal(z;\lambda,A_{\per}) \pm \eta + S(z;A)\big)\Vcal_{\pm},
\end{equation}
where 
\begin{equation*}
    S(z;A)=\begin{pmatrix}
    0 & 0 \\
    -2i\big(A-A_{\per})\partial + 2A_{\per}\big(A-A_{\per}\big) + \big(A-A_{\per}\big)^2 - i\big(A-A_{\per}\big)'_{\varphi} & 0 
\end{pmatrix},
\end{equation*}
or, equivalently
\begin{equation*}
     \Tcal_{\pm\eta}\Vcal_{\pm}=S(z;A)\Vcal_{\pm}.
\end{equation*}
\begin{lemma}\label{lem:exponential-shift}
    Assume that \ref{it:A.1}--\ref{it:A.4} are satisfied. Then there exists \(\delta = \delta(\lambda)\) such that for any \(\eta\in (0,\delta)\) the system \eqref{eq:exponential-shift} possesses exponential dichotomies on \(\R_+\) and \(\R_-\), respectively.
    Moreover, we may choose the dichotomies so that $\ker P(0;\lambda,A) = \ker P(0;\lambda_0,A_{per})$ for all $(\lambda,A)$ in a neighborhood of $(\lambda_0,A_0)$. The corresponding evolution operators are denoted by $\Psi^s(\cdot;\lambda,A)$ and $\Psi^{cu}(\cdot;\lambda,A)$ for $\Vcal_+$ on $\mathbb{R}_+$, etc. Moreover, the evolution operators depend smoothly on $\lambda$ and $A$.
\end{lemma}

\begin{proof}
    The proof follows from the non-autonomus generalization of Theorem ~1 mentioned in \cite{peterhof1997}. In fact, the proof is almost exactly the same as the proof of this theorem, with $e^{-A_-(t-\tau)}P_-$ and $e^{A_+(t-\tau)}P_+$ replaced by the evolution operators of the unperturbed system \eqref{eq:exponential-shift-periodic} here denoted by $\Omega^s(z,z_0)$ and $\Omega^{cu}(z,z_0)$ for $+\eta$ and $\Omega^u(z,z_0)$ and $\Omega^{cs}(z,z_0)$ for $-\eta$. Following the strategy of \cite[Lem. ~3.2]{peterhof1997}, we define the operator $\widetilde{T}_0:\mathscr{X}_{z_0}^s\to \mathscr{X}_{z_0}^s$ of the form $\widetilde{T}_0= I + I_1+I_2$, where $I$ is the identity operator and $I_1$ and $I_2$ are the integral operators
    \begin{equation*}
        \begin{aligned}
            (I_1\Vcal^s)(z) &= -\int_0^z\Omega^s(z,\xi)S(\xi)\Vcal^s(\xi)\dd\xi \\
            (I_2\Vcal^s)(z) &= \int_z^{\infty}\Omega^{cu}(z,\sigma)S(\xi)\Vcal^s(\xi)\dd\xi
        \end{aligned}
    \end{equation*}
    and
    \begin{equation*}
        \mathscr{X}_{z_0}^s=\{\Vcal\in C^0\left([z_0,\infty),X\right)\, : \, \lvert \Vcal\rvert_{\mathscr{X}_{z_0}^s}\coloneqq \sup_{z\geq z_0}e^{\eta\lvert z-z_0\rvert}\lvert \Vcal(z)\rvert<\infty\}.
    \end{equation*}
    For any $z_*$ we may decompose $I_1=S_1+K_1$ according to
    \begin{equation*}
        (K_1\Vcal^s)(z)\coloneqq\begin{cases}
            -\int_0^z\Omega^s(z,\xi)S(\xi)\Vcal^s(\xi)\dd\xi, & z\leq z_*, \\
            -\Omega^s(z,z_*)\int_0^{z_*}\Omega^s(z_*,\xi)S(\xi)\Vcal^s(\xi)\dd\xi, & z\geq z_*,
        \end{cases}
    \end{equation*}
    \begin{equation*}
        (S_1\Vcal^2)(z) = \begin{cases}
            0, & z\leq z_*, \\
            -\int_{z_*}^z\Omega^s(z,\xi)S(\xi)\Vcal^s(\xi)\dd\xi, & z\geq z_*.
        \end{cases}
    \end{equation*}
    Now, we just have to prove that the operator $K_1:\mathscr{X}_{z_0}^s\to \mathscr{X}_{z_0}^s$ is compact and that $S_1$ has sufficiently small norm. Note that for large $z_*$ we have, just as in \cite{peterhof1997}, that
    \[\lVert S_1 \rVert_{\mathcal{L}(\mathscr{X}_{z_0}^s)}\leq C\sup_{z\geq z_*}\lVert S(z)\rVert_{\mathcal{L}(X)} \leq C\varepsilon.\]
    Since $\varepsilon$ can be chosen arbitrarily small for $z_*$ large enough, we are done and it remains now only to prove that $K_1$ is compact. To this end, we restrict $K_1V^s$ to the interval $[0,z_*]$. $K_1$ maps $\mathscr{X}_{z_0}^s$ continuously into $C^1\left([0,z_*],D(\Bcal)\right)\cap C\left([0,z_*],X\right)$ and since $D(\Bcal)$ is compactly embedded into $X$, it follows by Arzela--Ascoli that this is compactly embedded into $C^0([0,z_*], X)$. Thus, the operator $K_1$ is a compact operator as the composition of a compact and a bounded operator. The rest follows in a completely similar way as in \cite{peterhof1997}, after pointing out that also the condition (H5) of the paper \cite{peterhof1997} is fulfilled by Theorem 2.5 in \cite{mielke2002}. Smoothness of the exponential dichotomies with respect to the parameters $(\lambda,A)$ can be shown by applying the implicit function theorem in a neighborhood of a point $(\lambda_0,A_0)$. 
\end{proof}

Reversing the shift, we obtain evolution operators for the system \eqref{eq:eigenvalue-spatial-dynamics} \(\Phi^s\) and \(\Phi^{cu}\) on \(\R_+\) and \(\Phi^u\) and \(\Phi^{cs}\) on \(\R_-\) given by

\begin{equation*}
    \begin{split}
        \Phi^s(z,z_0;\lambda,A) = e^{-\eta(z-z_0)}\Psi^s(z,z_0;\lambda,A), \quad & \Phi^{cu}(z,z_0;\lambda,A) = e^{-\eta(z-z_0)}\Psi^{cu}(z,z_0;\lambda,A), \\
        \Phi^u(z,z_0;\lambda,A) = e^{\eta(z-z_0)}\Psi^u(z,z_0;\lambda,A), \quad & \Phi^{cs}(z,z_0;\lambda,A) = e^{\eta(z-z_0)}\Psi^{cs}(z,z_0;\lambda,A).
    \end{split}
\end{equation*}
        
\section{Exponential decay of eigenfunctions}\label{sec:exponentialdecay}

\begin{lemma}\label{lem:sol-rep}
    Let $\Ucal$ be a solution of \eqref{eq:eigenvalue-spatial-dynamics}. Then
    \begin{enumerate}[(i)]
        \item if $\Ucal$ is bounded on $\mathbb{R}_+$, then for every $R\geq 0$ there exists a $\Ucal_0^s \in X^s$ and a $\Ucal_0^c\in X^c$ such that for $z\geq R$,
        \begin{equation}\label{eq:lem41_+}
            \begin{aligned}
                \Ucal(z) = \Phi^s(z,R;\lambda, A_{\per})\Ucal_0^s &+ \Phi^c(z,0;\lambda, A_{\per})\Ucal_0^c + \int_R^z\Phi^s(z,\xi;\lambda, A_{\per})S(\xi;A)\Ucal(\xi)d\xi \\
                &-\int_z^{\infty}\Phi^{cu}(z,\xi;\lambda,A_{\per})S(\xi;A)\Ucal(\xi)d\xi,
            \end{aligned}
        \end{equation}
        \item if $\Ucal$ is bounded on $\mathbb{R}_-$, then for every $R\geq 0$ there exists a $\Vcal_0^u \in X^u$ and a $\Vcal_0^c\in X^c$ such that for $z\leq -R$,
        \begin{equation}\label{eq:lem41_-}
            \begin{aligned}
                \Ucal(z) = \Phi^u(z,-R;\lambda,A_{\per})\Vcal_0^s &+ \Phi^c(z,0;\lambda,A_{\per})\Vcal_0^c - \int_z^{-R}\Phi^u(z,\xi;\lambda,A_{\per})S(\xi;A)\Ucal(\xi)d\xi \\
                &+\int_{-\infty}^z\Phi^{cs}(z,\xi;\lambda,A_{\per})S(\xi;A)\Ucal(\xi)d\xi.
            \end{aligned}
        \end{equation}
    \end{enumerate}
\end{lemma}
\begin{proof}
    We consider the case when $\Ucal$ is bounded on $\R_+$, since the other case is completely identical. We introduce $\Ucal^i \coloneqq P^i(\cdot;\lambda,A)\Ucal$ for $i\in\{s,c,u\}$. By the variation of constants formula, after applying projections, we get
    \begin{equation}
        \begin{aligned}
            \Ucal^s(z) &= \Phi^s(z,\zeta;\lambda,A_{\per})\Ucal^s(\zeta)+\int_{\zeta}^z\Phi^s(z,\xi;\lambda,A_{\per})S(\xi;A)\Ucal(\xi)d\xi, \quad z\geq\zeta\geq 0, \\
            \Ucal^c(z) &= \Phi^c(z,\zeta;\lambda,A_{\per})\Ucal^c(\zeta)+\int_{\zeta}^z\Phi^c(z,\xi;\lambda,A_{\per})S(\xi;A)\Ucal(\xi)d\xi, \quad z,\zeta\geq 0, \\
            \Ucal^u(z) &= \Phi^s(z,\zeta;\lambda,A_{\per})\Ucal^u(\zeta)+\int_{\zeta}^z\Phi^u(z,\xi;\lambda,A_{\per})S(\xi;A)\Ucal(\xi)d\xi, \quad \zeta\geq z\geq 0.
        \end{aligned}
        \label{eq:projeqn}
    \end{equation}
    Since $\lVert \Ucal(z) \rVert_X$ is bounded, it follows that also $\Ucal^s,\Ucal^c$ and $\Ucal^u$ all are bounded.

    We first look at $\Ucal^u$ and let $\zeta \to \infty$ in the last equation of \eqref{eq:projeqn}. Since $\Ucal^u$ is bounded, the first term goes to zero, and it follows that 
    \begin{equation*}
        \Ucal^u(z) = - \int_z^{\infty}\Phi^u(z,\xi;\lambda,A_{\per})S(\xi;A)\Ucal(\xi)d\xi.
    \end{equation*}
    Next, we study the equation for $\Ucal^c$. 
    
    Note that $\Phi^c(z,\xi;\lambda,A_{\per})$ is a bounded operator mapping a finite-dimensional set into itself. Therefore
    \begin{equation*}
\Phi^c(z,\xi;\lambda,A_{\per})=\Phi^c(z,\xi;\lambda,A_{\per})\Phi^c(\xi,0;\lambda,A_{\per})\Phi^c(0,\xi;\lambda,A_{\per}) = \Phi^c(z,0;\lambda,A_{\per})\Phi^c(0,\xi;\lambda,A_{\per}).
    \end{equation*}
    Now it follows that the integral on the right-hand side of the second equation of \eqref{eq:projeqn} converges as $\zeta \to \infty$ since $\lVert \Phi^c(0,\xi;\lambda,A_{\per})\rVert$ is uniformly bounded for all $\xi\in\mathbb{R}$ and since 
    \begin{equation}
    \begin{aligned}
    \int_z^{\infty}\lVert S(\xi;A)\Ucal(\xi)\rVert_X d\xi &\leq \lVert S(\cdot;A)\rVert_{X_{\beta}}\lVert \Ucal \rVert_{L^{\infty}}\int_z^{\infty}\dfrac{1}{(1+\xi)^{\beta}} \dd \xi \\
    &=\dfrac{1}{\beta - 1}\lVert S(\cdot;A)\rVert_{X_{\beta}}\lVert \Ucal \rVert_{L^{\infty}}\dfrac{1}{(1+z)^{\beta - 1}},
    \end{aligned}
    \label{eq:intest}
    \end{equation}
    where we recall the assumption that $\beta >1$. Since the left-hand side of the equation for $\Ucal^c$ does not depend on $\xi$, it follows that $\lim_{\zeta\to \infty}\Phi^c(0,\xi;\lambda,A_{\per})\Ucal^c(\xi)$ exists. Hence, there exists a $\Ucal_0^c\in X^c$ such that 
    \begin{equation*}
        \Ucal^c(z)= \Phi^c(z,0;\lambda,A_{\per})\Ucal_0^c - \int_z^{\infty}\Phi^c(z,\xi;\lambda,A_{\per})S(\xi,A)\Ucal(\xi) d\xi.
    \end{equation*}
    For $\Ucal^s$, we choose $\xi=R\geq 0$ arbitrarily, so that by \eqref{eq:projeqn} for $z\geq R$,
    \begin{equation*}
        \begin{aligned}
            \Ucal(z) = \Phi^s(z,R;\lambda,A_{\per})\Ucal_0^s &+ \Phi^c(z,0;\lambda,A_{\per})\Ucal_0^c + \int_R^z\Phi^s(z,\xi;\lambda,A_{\per})S(\xi;A)\Ucal(\xi)d\xi \\
                &-\int_z^{\infty}\Phi^{cu}(z,\xi;\lambda,A_{\per})S(\xi;A)\Ucal(\xi)d\xi,
        \end{aligned}
    \end{equation*}
    with $\Ucal_0^s=\Ucal^s(R)$.
\end{proof}

\begin{lemma}\label{lem:exp-dec}
     Assume that assumptions \ref{it:A.1}-\ref{it:A.4} are satisfied. Further assume that $\lambda > 0$ is an eigenvalue of the operator $L_A$ with corresponding eigenfunction $u\in H^2(\mathbb{R}\times S^1)$. We denote by $\Ucal$ the corresponding solution of the system \eqref{eq:eigenvalue-spatial-dynamics}. Let $\eta\in\big(0,\delta(\lambda)\big)$. Then there exists a constant $K>0$ such that 
    \begin{equation*}
        \lVert \Ucal(z) \rVert_X \leq Ke^{-\eta\lvert z\rvert},
    \end{equation*}
    for all $z\in\mathbb{R}$.
\end{lemma}
\begin{proof}
    The proof of this result follows by the application of Banach's fixed point theorem in exponentially weighted spaces. Since, in light of Lemma \ref{lem:sol-rep}, the proof follows exactly as Lemma 6 in \cite{laptev2017}, we only give a brief summary of the precise method. We only consider the exponential decay for the case $z\to + \infty.$ The case $z\to -\infty$ follows identically relying on Lemma \ref{lem:sol-rep} (ii).
    
    Recall that by equation \ref{eq:lem41_+} for $z\geq R$, $\Ucal$ is a solution to the fixed-point problem
    \begin{equation*}
        \begin{aligned}
                \Ucal(z) = F(\Ucal)(z) := \Phi^s(z,R;\lambda, A_{\per})\Ucal_0^s & + \Phi^c(z,0;\lambda, A_{\per})\Ucal_0^s + \int_R^z\Phi^s(z,\xi;\lambda, A_{\per})S(\xi;A)\Ucal(\xi)d\xi \\
                &-\int_z^{\infty}\Phi^{cu}(z,\xi;\lambda,A_{\per})S(\xi;A)\Ucal(\xi)d\xi \\ 
                = \Phi^s(z,R;\lambda, A_{\per})\Ucal_0^s & + \Phi^c(z,0;\lambda, A_{\per})\Ucal_0^s +  I_1(z) + I_2(z)
            \end{aligned}
    \end{equation*}
    First we show that $\Ucal_0^c=0$. Relying on standard semigroup estimates, one may argue that $I_1(z),I_2(z) \to 0$ as $z\to +\infty$. Since $\Phi^s(z,R;\lambda, A_{\per})\Ucal_0^s$ decays exponentially by construction, we find that $\Ucal(z) \to \Ucal_0^c$ as $z\to +\infty$. But since $\Ucal\in H^1(\R;X)$, it necessarily follows $\Ucal_0^c=0$.

    The exponential decay now follows by showing that $F$ is a contraction in the space $Y_{\eta} = \{\Vcal \in C([R;\infty);\Lcal(X)) \,:\, \|\Vcal\|_{Y_{\eta}} := \sup_{z\geq R} e^{\eta z} \|\Vcal(z)\|_X < \infty\}$. This follows by applying standard semigroup estimates and precisely in the same fashion as in the proof of \cite[Lemma 6]{laptev2017}, provided that one chooses $R$ large enough.
\end{proof}

\section{Lyapunov--Schmidt reduction}\label{sec:Lyapunov}
We are now ready to prove our main result. For this, we utilize the Lyapunov-Schmidt reduction. We denote by $u_*$ the eigenfunction corresponding to the embedded eigenvalue of the unperturbed problem, i.e., $L_{A_0}u_*=\lambda_0u_*$. We assume, without loss of generality, that $u_*$ is normalized. We further write $\Ucal_*=(u_*,u_*')^T$ as the corresponding solution of \eqref{eq:eigenvalue-spatial-dynamics} with $A=A_0$. Using the evolution operators and their related projections from the exponential dichotomy, we define the stable and unstable subspaces $E_+^s$ and $E_-^u$ by
\begin{equation*}
    \begin{aligned}
        E_+^s &= \{\Ucal\in Y; \hspace{2mm} P_+^s(0;\lambda_0,A_0)\Ucal=\Ucal\}, \\
        E_-^u &= \{\Ucal\in Y; \hspace{2mm} P_-^u(0;\lambda_0,A_0)\Ucal=\Ucal\}.
    \end{aligned}
\end{equation*}
These subspaces consist of inital values of solutions of the unperturbed system which decay exponentially as $z\to\infty$ and $z\to -\infty$, respectively. We have that $E_+^s\cap E_-^u = \text{span}\{\Ucal_*(0)\}$, since $\lambda_0$ was assumed to be of multiplicity $1$. The following is a corollary of Lemma \ref{lem:exp-dec}.
\begin{corollary}
    Let $\Ucal$ be a solution of \eqref{eq:eigenvalue-spatial-dynamics}. If $\Ucal(0)\in E_+^s\cap E_-^u$, then $\Ucal_0^c=\Vcal_0^c$ in $(i)$ and $(ii)$ of Lemma \ref{lem:sol-rep}.
\end{corollary}
To find embedded eigenvalues, we define the mapping $\iota\colon E_+^s\times E_-^u\times \mathbb{R}\times (A_{\per} + \Xcal_{\beta})\to X$ by
$$\iota\colon (\Ucal_0^s,\Ucal_0^u;\lambda,A)=P_+^s(0;\lambda,A)\Ucal_0^s-P_-^u(0;\lambda,A)\Ucal_0^u.$$
\begin{lemma}
    It holds that $\lambda\notin\sigma_{per}(L_{A_{\per}})$ is an eigenvalue of $L_A$ if and only if there exist $\Ucal_0^s\in E_+^s$ and $\Ucal_0^u\in E_-^u$ with $(\Ucal_0^s,\Ucal_0^u)\neq (0,0)$ such that
    \begin{equation}
    \iota(\Ucal_0^s,\Ucal_0^u;\lambda,A)=0.
    \label{eq:ieq}
    \end{equation}
\end{lemma}
\begin{proof}
    The proof of the result follows along the same lines as that of \cite[Lemma 7]{laptev2017}. For the reader's convenience, we reproduce the argument here. Assume that \ref{eq:ieq} holds so that we know $P_+^s(0;\lambda,A)\Ucal_0^s = P_-^u(0;\lambda,A)\Ucal_0^u \neq 0$. Then the solution $\Ucal$ to \ref{eq:eigenvalue-spatial-dynamics} with initial condition $P_+^s(0;\lambda,A)\Ucal_0^s = P_-^u(0;\lambda,A)\Ucal_0^u$ must have exponential decay both for $z\to +\infty$ and $z\to -\infty$. But then, by Lemma \ref{lem:equivalence-spatial-dynamics} the first component of $\Ucal$ solves the eigenvalue problem for $L_A$ with eigenvalue $\lambda$.
    
    For the converse direction note that solutions to the eigenvalue problem $L_Au = \lambda u$ give rise to exponentially decaying solutions $\Ucal$ of the spatial dynamics problem and we may conclude that $\Ucal(0) = P_+^s(0;\lambda,A)\Ucal(0) = P_-^u(0;\lambda,A)\Ucal(0)$. This concludes the proof.
\end{proof}
To solve \eqref{eq:ieq}, we must first observe that for any $(\Ucal_0^s,\Ucal_0^u)\in E_+^s\times E_-^u$ we have 
$$\iota(\Ucal_0^s,\Ucal_0^u;\lambda_0,A_0)=\Ucal_0^s-\Ucal_0^u.$$
Thus, $\ran \,\iota(\cdot,\cdot;\lambda_0,A_0)=E_+^s+E_-^u$. Next, we will provide a result on the codimension of $E_+^s+E_-^u$.
\begin{lemma}\label{lem:codim-E}
    We have that $\codim(E_+^s+E_-^u)=\dim (X^c) + 1 = 2m+1$.
\end{lemma}
\begin{proof}
In light of the constructions in Lemma \ref{lem:exponential-shift}, the proof follows in a completely similar way as the proof of Lemma 8 in \cite{laptev2017}, just with $4m-2$ substituted with $\dim(X^c)=2m$ everywhere and with the constructions of $S_j$ and $K_j$ according to Lemma \ref{lem:exponential-shift}. 
\end{proof}
Let $Q$ be the projection in $X$ onto $\text{Ran}\iota(\cdot,\cdot;\lambda_0,A_0)=E_+^s+E_-^u$. Lemma \ref{lem:codim-E} implies that $\dim(\ker Q)=\dim(X^c)+1=2m+1.$ Equation \eqref{eq:ieq} can be rewritten as 
\begin{equation}
    \begin{aligned}
        Q\iota(\Ucal_0^s,\Ucal_0^u;\lambda,A)&=0, \\
        (I-Q)\iota(\Ucal_0^s,\Ucal_0^u;\lambda,A)&=0.
    \end{aligned}
    \label{eq:projieq}
\end{equation}
Next, we want to add a condition to fix the solution amongst infinitely many to get a unique solution using the implicit function theorem.
\begin{lemma} \label{lem:hyperplane}
    Let $D\subset E_+^s\times E_-^u$ be an affine hyperplane such that 
    $$D\cap \text{span}\big\{\big(\Ucal_*(0),\Ucal_*(0)\big)\big\}=\big\{\big(\Ucal_*(0),\Ucal_*(0)\big)\big\}.$$
    For $(\lambda,A)$ close to $(\lambda_0,A_0)$, the first equation of \eqref{eq:projieq} has a unique solution
    $$(\Ucal_0^s,\Ucal_0^u)=\big(\Ucal_0^s(\lambda,A),\Ucal_0^u(\lambda,A)\big)\in D$$
    in a neighborhood of $\big((\Ucal_*(0),\Ucal_*(0)\big)$. Moreover, $\Ucal_0^s$ and $\Ucal_0^u$ are smooth in their arguments.
\end{lemma}
\begin{proof}
    This follows immediately from \cite[Lem.~9]{laptev2017} by the implicit function theorem.
\end{proof}
Using the variation of constants formula, see e.g. \cite{henry1981}, we get that 
\begin{equation*}
\begin{aligned}
    \iota(\cdot,\cdot;\lambda,A) &= \Ucal_0^s-\Ucal_0^u \\
    +&\int_{-\infty}^0\Phi^{cs}(0,z;\lambda_0,A_0)((\lambda-\lambda_0)N+\tilde{S}(z;A))\Phi^u(z,0;\lambda,A)\Ucal_0^u \dd z \\
    +& \int_0^{\infty}\Phi^{cu}(0,z;\lambda_0,A_0)((\lambda-\lambda_0)N+\tilde{S}(z;A))\Phi^s(z,0;\lambda,A)\Ucal_0^s \dd z,
\end{aligned}
\end{equation*}
where $N:X\to X$ is given by
\begin{equation}
    N=\begin{pmatrix}
        0 & 0 \\
        1 & 0
    \end{pmatrix}
    \label{eq:N-def}
\end{equation}
and $\tilde{S}$ is defined exactly as in \eqref{eq:perturb-shift}, but with $A_{\per}$ replaced with $A_0$, as we are now perturbing from the unperturbed system and not the system at infinity.

To solve the second equation of \eqref{eq:projieq}, we define $F:\R\times(A_{\per}+\Xcal_{\beta}) \to \ker Q$ by 
\begin{equation}
    \begin{aligned}
        F(\lambda,A) &= -\iota(\Ucal_0^s(\lambda,A),\Ucal_0^u(\lambda,A);\lambda,A) = -(I-Q)\iota(\Ucal_0^s(\lambda,A),\Ucal_0^u(\lambda,A);\lambda,A) \\
        &= \int_{-\infty}^0(I-Q)\Phi^{cs}(0,z;\lambda_0,A_0)((\lambda-\lambda_0)N+\tilde{S}(z;A))\Phi^u(z,0;\lambda,A)\Ucal_0^u \dd z \\
    &\quad + \int_0^{\infty}(I-Q)\Phi^{cu}(0,z;\lambda_0,A_0)((\lambda-\lambda_0)N+\tilde{S}(z;A))\Phi^s(z,0;\lambda,A)\Ucal_0^s \dd z.
    \end{aligned}
    \label{eq:F-func}
\end{equation}
Note in particular that $F$ is a smooth function in both arguments and that solving \eqref{eq:projieq} is equivalent to solving $F(\lambda,A)=0$. To solve the equation $F(\lambda,A)=0$, we use the dual space $X^*=H^{-1}(S^1)\times L^2(S^1)$ of $X$ and the adjoint equation for $\lambda=\lambda_0$ as well as $A=A_0$, that is the equation
\begin{equation}
    \Wcal' = -(\Bcal(z;\lambda_0,A_{\per}) + S(z; A_0))^*\Wcal.
    \label{eq:adjoint-eq}
\end{equation}
The adjoint system possesses exponential dichotomies on $\R_+$ and $\R_-$ denoted by $\Psi^s(z,z_0;\lambda_0,A_0)$, $\Psi^{cu}(z,z_0;\lambda_0,A_0)$ and $\Psi^{cs}(z,z_0;\lambda_0,A_0)$, $\Psi^u(z,z_0;\lambda_0,A_0)$, respectively. They are related to the exponential dichotomies of the unperturbed system in \eqref{eq:eigenvalue-spatial-dynamics}, with $\lambda=\lambda_0$ and $A=A_0$,  in the following way
 \begin{equation*}
     \begin{aligned}
     \Psi^s(z,z_0;\lambda_0,A_0) = \Phi^{cu}(z_0,z;\lambda_0,A_0)^*, \quad  \Psi^{cu}(z,z_0;\lambda_0,A_0) = \Phi^{s}(z_0,z;\lambda_0,A_0)^*, \\
     \Psi^{cs}(z,z_0;\lambda_0,A_0) = \Phi^{u}(z_0,z;\lambda_0,A_0)^*, \quad  \Psi^u(z,z_0;\lambda_0,A_0) = \Phi^{cs}(z_0,z;\lambda_0,A_0)^*,
     \end{aligned}
 \end{equation*}
 where $z_0$, $z$ are in the appropriate intervals. It is straightforward to verify that $\Ucal_*^{\perp} = (-u_*',u_*)^T$ solves the adjoint equation \eqref{eq:adjoint-eq} and is exponentially decaying as $\lvert z\rvert \to \infty$. For $\Ucal^s\in E_+^s$ and $\Ucal^u\in E_-^u$ it is easy to see that 
 \begin{equation*}
    \begin{aligned}
    \dfrac{d}{dz}\langle \Ucal_*^{\perp}(z), \Phi^s(z,0;\lambda_0,A_0)\Ucal^s \rangle &= \dfrac{d}{dz}\langle \Psi^s(z,0;\lambda_0,A_0)\Ucal_*^{\perp}(0), \Phi^s(z,0;\lambda_0,A_0)\Ucal^s \rangle = 0, \\
    \dfrac{d}{dz}\langle \Ucal_*^{\perp}(z), \Phi^u(z,0;\lambda_0,A_0)\Ucal^u \rangle &= \dfrac{d}{dz}\langle \Psi^u(z,0;\lambda_0,A_0)\Ucal_*^{\perp}(0), \Phi^u(z,0;\lambda_0,A_0)\Ucal^u \rangle = 0.
    \end{aligned}
\end{equation*}
Moreover, since $\Ucal_*^{\perp}(z) \to 0$ as $\lvert z \rvert \to \infty$ and $\Phi^s(z,0;\lambda_0,A_0)\Ucal^s$ and $\Phi^u(z,0;\lambda_0,A_0)\Ucal^u$ are bounded on $\R_+$ and $\R_-$ respectively, it follows that $\langle \Ucal_*^{\perp}(0),\Ucal^s + \Ucal^u\rangle=0$. This also implies that for any $\Ucal\in X$ we have that $\langle \Ucal_*^{\perp}(0),Q\Ucal\rangle = \langle Q^*\Ucal_*^{\perp}(0),\Ucal\rangle = 0$, and so it follows that $\Ucal_*^{\perp}(0)\in\ker Q^*$.

\begin{lemma}\label{lem:lambda-der}
    The equation $\langle \Ucal_*^{\perp}(0), F(\lambda,A)\rangle = 0$ defines a smooth function $\lambda(A)$ in a neighborhood of $A=A_0$ such that $\lambda(A_0)=\lambda_0$. Furthermore, for any $B\in \Xcal_{\beta}$, 
    \begin{equation}
        \lambda'(A_0)B=-\int_{-\infty}^{\infty}\langle u_*(z),(-2iB(z,\cdot)\partial - iB_{\phi}'(z,\cdot) + 2A_0(z,\cdot)B(z,\cdot))u_*(z)\rangle_{L^2(S^1)} \dd z.
        \label{eq:lambda_expr}
    \end{equation}
\end{lemma}
\begin{proof}
    We have that
    \begin{equation*}
        \begin{aligned}
            0 &= F_*(\lambda,A)\coloneqq \langle\Ucal_*^{\perp}(0),F(\lambda,A)\rangle \\
            &= \int_{-\infty}^0 \langle \Ucal_*^{\perp}(0),\Phi^{cs}(0,z;\lambda_0,A_0)((\lambda-\lambda_0)N + \tilde{S}(z;A))\Phi^u(z,0;\lambda,A)\Ucal_0^u(\lambda,A)\rangle \dd z \\
            &\quad +\int_0^{\infty}\langle\Ucal_*^{\perp}(0),\Phi^{cu}(0,z;\lambda_0,A_0)((\lambda-\lambda_0)N+\tilde{S}(z;A))\Phi^s(z,0,\lambda,A)\Ucal_0^s(\lambda,A)\rangle \dd z \\
            &= \int_{-\infty}^0 \langle \Ucal_*^{\perp}(z),((\lambda-\lambda_0)N + \tilde{S}(z;A))\Phi^u(z,0;\lambda,A)\Ucal_0^u(\lambda,A)\rangle \dd z \\
            &\quad +\int_0^{\infty}\langle\Ucal_*^{\perp}(z),((\lambda-\lambda_0)N+\tilde{S}(z;A))\Phi^s(z,0,\lambda,A)\Ucal_0^s(\lambda,A)\rangle \dd z.
        \end{aligned}
    \end{equation*}
    By Lemma \ref{lem:exponential-shift} it follows that $F_*$ is a smooth function of $\lambda$ and $A$ in a neighborhood of $(\lambda_0,A_0)$. Moreover, $F_*(\lambda_0,A_0)=0$ and
\begin{equation*}
    \dfrac{\partial F_*}{\partial\lambda}(\lambda_0,A_0)=\int_{-\infty}^{\infty}\langle\Ucal_*^{\perp}(z),N\Ucal_*(z)\rangle \dd z = \int_{-\infty}^{\infty}\lVert u_*(z)\rVert_{L^2(S^1)}^2 \dd z = 1,
\end{equation*}
by the normalization of $u_*$. This allows for the application of the implicit function theorem to solve for $\lambda$ giving it as a function of $A$ in a neighborhood of $A=A_0$ with $\lambda(A_0)=\lambda_0$. Differentiation of the function $F_*(\lambda(A),A)=0$ and evaluating in $A=A_0$ gives the expression in \eqref{eq:lambda_expr}. For more details on the computation of the derivatives, see \cite[Lem. ~4.4]{maadsasane2024} as it follows similarly.
\end{proof}
We are now ready to prove Theorem \ref{thm:main}.
\begin{proof}[Proof of Theorem \ref{thm:main}]
Since $\dim(\ker Q^*)=2m+1$, and we know that $\Ucal_*^{\perp}(0)\in\ker Q$, we define $\Wcal_k(0)\in X'$, k=1,\dots,2m, to be such that $\{\Wcal_k(0);\hspace{2mm} k=1,\dots,2m\}\cup\{\Ucal_*^{\perp}(0)\}$ is a basis for $\ker Q^*$. Let $\Wcal_k$ be the solutions of the unperturbed system with initial value $\Wcal_k(0)$ and define $F_k:A_{\per} +\Xcal_{\beta}\to \R$ by
\[F_k(A)=\langle\Wcal_k(0),F(\lambda(A),A)\rangle, \quad \text{for } k=1,\dots,2m,\]
with $F$ as in \eqref{eq:F-func}. All $F_k$ are clearly smooth, since $F$ is by Lemma \ref{lem:exponential-shift}. Further if $F_k(A)=0$ for some $A\in A_{\per} +\Xcal_{\beta}$ for all $k=1,\dots,2m$, then $F(\lambda(A),A)=0$ since $\{\Wcal_k(0);\hspace{2mm} k=1,\dots,2m\}\cup\{\Ucal_*^{\perp}(0)\}$ constitutes a basis for $\ker Q^*$. The converse clearly holds as well.

In order to prove the theorem, we show that there is a manifold of perturbations $A$ with codimension $2m$ defined by the equations $F_k(A)=0$ for $k=1,\dots,2m$. Using the notation
\begin{equation*}
    \Ucal(z;A) \coloneqq \begin{cases}
        \Phi^s(z,0;\lambda(A),A)\Ucal_0^s(\lambda(A),A), \quad \text{for } z\geq 0, \\
        \Phi^u(z,0;\lambda(A),A)\Ucal_0^u(\lambda(A),A), \quad \text{for } z< 0,
    \end{cases}
\end{equation*}
we get that
\begin{equation*}
\begin{aligned}
F_k&(A) = \int_{-\infty}^\infty \langle \Wcal_k(z), ((\lambda(A) - \lambda_0)N + \tilde{S}(z; A))\Ucal(z; A) \rangle dz \\
&= \int_{-\infty}^\infty \langle w_k(z), (\lambda(A) - \lambda_0 -2i\big(A-A_{0})\partial + (A_{0}+A)(A-A_{0}) - i\big(A-A_{0}\big)_{\varphi}')u(z; A) \rangle_{L^2(S^1)} dz,
\end{aligned}
\end{equation*}
where $W_k = (-w_k', w_k)^T$, $u(z; A)$ is the first component of $U(z; A)$ and $N$ as in \eqref{eq:N-def}.

Next, we wish to prove that $F_k'(A_0)$, $k = 1,\dots, 2m$ are linearly independent. Using \eqref{lem:hyperplane} and Lemma \ref{lem:lambda-der}, we get that, for every $B\in \Xcal_{\beta}$
\begin{equation*}
    \begin{aligned}
        F_k'(A_0)B = 
        \int_{-\infty}^{\infty}\langle w_k(z),([\lambda'(A_0)B]-  2iB\partial+2A_0B-iB_{\varphi}')u_*(z)\rangle_{L^2(S^1)}\dd z
    \end{aligned}
\end{equation*}
with $\lambda'(A_0)B$ as in $\eqref{eq:lambda_expr}$. Now assume that there exists $\alpha_1,\dots,\alpha_m\in\C$ such that 
\begin{equation*}
    \begin{aligned}
        0 &= \sum_{k=1}^{2m}\alpha_kF_k'(A_0)B \\
        &= \int_{-\infty}^{\infty}\langle w(z),([\lambda'(A_0)B]-  2iB\partial+2A_0B-iB_{\varphi}')u_*(z)\rangle_{L^2(S^1)}\dd z \\
        &= \int_{-\infty}^{\infty}\langle w(z)+ \alpha_*u_*(z),(-  2iB\partial+2A_0B-iB_{\varphi}')u_*(z)\rangle_{L^2(S^1)}\dd z,
    \end{aligned}
\end{equation*}
for every $B\in \Xcal_{\beta}$. Here,
\begin{equation*}
    w=\sum_k^{2m}\alpha_kw_k,
\end{equation*}
together with
\begin{equation*}
    \alpha_* \coloneqq \int_{-\infty}^{\infty}\overline{\langle w(z),u_*(z)\rangle_{L^2(S^1)}}\dd z.
\end{equation*}
Using integration by parts in the $\varphi$-variable, this becomes 
\begin{equation*}
    \begin{aligned}
        \int_{-\infty}^{\infty}&\int_{S^1}\Big(i\partial_{\varphi}\left(u_*(z,\varphi)\overline{(w(z,\varphi)+\alpha_*u_*(z,\varphi))}\right) - 2iu'_{\varphi}(z,\varphi)\overline{(w(z,\varphi)+\alpha_*u_*(z,\varphi))} \\
        &+ 2A_0(z,\varphi)u_*(z,\varphi)\overline{(w(z,\varphi)+\alpha_*u_*(z,\varphi))}\Big)B(z,\varphi)\dd\varphi\dd z,
    \end{aligned}
\end{equation*}
for every $B\in \Xcal_{\beta}$. Since $\Xcal_{\beta}$ is dense in $L^2(\R\times S^1;\R)$, we get that, after rewriting
\begin{equation}
\begin{aligned}
u_*(z,\varphi)&\overline{(w(z,\varphi)+\alpha_*u_*(z,\varphi))} \\ 
&- \left(u_{*\varphi}'(z,\varphi) + 2iA_0(z,\varphi)u_*(z,\varphi)\right)\overline{(w(z,\varphi)+\alpha_*u_*(z,\varphi))}=0,
\end{aligned}
\label{eq:zero-eq}
\end{equation}
for every $(z,\varphi)\in\R\times S^1$, since $u_*$ and $w$ are continuously differentiable. We now wish to prove that $w+\alpha_*u_*\equiv 0$. This will be done by first proving that the function vanishes on an open set, followed by unique continuation, yielding the result. We know by \ref{it:A.2} that there exists a $z_0\in\R$ such that $\int_{S^1}A_0(z_0,\varphi)\dd\varphi \neq \int_{S^1}A_{\per}(z_0,\varphi)\dd\varphi$. 

However, since clearly 
\begin{equation*}
    \left\lvert\int_{S^1}A_0(z,\varphi)\dd\varphi - \int_{S^1}A_{\per}(z,\varphi)\dd\varphi\right\rvert \longrightarrow 0
\end{equation*}
as $\lvert z\rvert\to\infty$, it follows by continuity and the intermediate value theorem that we can choose $z_0$ in the way we want. By continuity, there exists an open interval $I$ such that $z_0\in I$ and $\int_{S^1}A_0(z_0,\varphi)\dd\varphi\neq k\pi$ for all $z\in I$.
We now wish to prove that either $w+\alpha_*u_*$ or $u_*$ vanishes on an open interval. This however now follows exactly in the same way as in \cite{laptev2017}. Then, by unique continuation, it follows that it has to vanish on the entire $\mathbb{R}\times S^1$. However $u_*$ is an eigenfunction by assumption, and so it cannot vanish everywhere, giving that $w+\alpha_*u_* \equiv 0$ on $\R\times S^1$. Thus, it follows that $\alpha_k=0$ for all $k=1,\dots,2m$, implying that $F_k'(A_0)$, $k=1,\dots,2m$ are linearly independent. The final result is then obtained in the following way:

Define $\mathcal{F}: A_{\per} + \Xcal_{\beta}\to \C^{2m}$ by
\[\mathcal{F}(A)\coloneqq \left(F_1(A),\dots,F_{2m}(A)\right)^T.\]
It follows by the linear independence of $\mathcal{F}'(A_0)$ that we can decompose the space $\Xcal_{\beta}$ in the following way
\[\Xcal_{\beta}=\left(\ker \mathcal{F}'(A_0) \right)\oplus \mathcal{M},\]
where $\mathcal{M}$ is $2m$-dimensional. The map $\mathcal{F}'(A_0)$ is clearly surjective. In light of the decomposition $\Xcal_{\beta}$, we write for any $B\in \Xcal_{\beta}$ $B=B_1+B_2$, where $B_1\in\ker\mathcal{F}'(A_0)$ 
and $B_2\in\mathcal{M}$. We may then define 
\[\mathcal{G}:\left(\ker \mathcal{F}'(A_0) \right)\oplus \mathcal{M}\to \C^{2m}\]
by 
\[\mathcal{G}(B_1,B_2)=\mathcal{F}(B_1+B_2+A_0).\]
The Fréchet derivative of $\mathcal{G}$ at $(0,0)$ with respect to $B_2$ denoted by $\partial_{A_2}\mathcal{G}(0,0)$ is invertible and hence allows for the application of the implicit function theorem, solving for $B_2$ in terms of $B_1$, defining a smooth manifold of dimension $2m$ in a neighborhood of $A=A_0$. For more explicit details on the construction, see the proof of Theorem 1.3 in \cite{maadsasane2024}, as it is identical. 
\end{proof}

\section{Example}\label{sec:ex}
The following example shows that there exist operators satisfying the conditions \ref{it:A.1}--\ref{it:A.4}. For simplicity, we consider an operator whose potential is independent of $\varphi$. We emphasize that the theorems in the paper do not rely on the $\phi$-independence of $A_0$. 
\begin{example}
   Let $A_{per}(z)=\cos z$. We will find a potential $A_0$ such that $\lvert A_0(z) - A_{per}(z)\rvert\to 0$ as $|z|\to \infty$.
   For simplicity, we will construct the embedded eigenvalue for a fairly small $\lambda_0$. Examples for other values of $\lambda_0$ can be constructed with the same method used here, but with more work if $\lambda_0$ is large. 
   The essential spectra of $A_0$ and $A_{per}$ on $L^2(\mathbb R\times S^1)$ are the same due to Weyl's theorem, and so we establish the essential spectrum of $A_{per}$ first. It is generated by functions of the form $v(z) e^{i m\varphi}$, $m\in \mathbb Z$, where
   $v$ is a bounded solution of
   \begin{equation}\label{E:example-ode}
      -v''(z)+(m+\cos z)^2 v(z) = \lambda v(z).
   \end{equation}
\end{example}
The essential spectrum is the closure of the set of numbers $\lambda$ for which such solutions exist. 
For a fixed $m$, we obtain bounded solutions $v$, when $\lambda$ belongs to the essential spectrum of the operator
\begin{equation*}
   \ell_m = -\dfrac{d^2}{dz^2} + (m + \cos z)^2
\end{equation*}
on $L^2(\mathbb R)$.
Such an operator has essential spectrum with a band structure, i.e. it consists of a countable union of closed intervals \cite[Thm.~5.3.2]{eastham1973} which is contained in a half-line $[\mu_m,\infty)$, where $\mu_m\in [\min(m+\cos z)^2,\max(m+\cos z)^2]$. In particular, $\mu_m\in [0,1]$ if $m=0$ while $\mu_m\in [(|m|-1)^2,(|m|+1)^2]$ if $m\ne 0$. The start and endpoints of each of these closed intervals can be computed more accurately numerically by computing the discriminant $\gamma(\lambda)=\frac{1}{2}(\phi_1(p)+\phi_2'(p))$, where $p$ is the period of the potential ($p=\pi$ for $m=0$ and $p=2\pi$ for $m\geq 1$) and $\phi_i$ are the two linearly independent solutions of the periodic ODE satisfying the initial conditions $\phi_1(0)=1$, $\phi_1'(0)=0$ and $\phi_2(0)=0$, $\phi_2'(0)=1$, respectively, and solving for the intersections of the graph of $\gamma$ with the lines $y=1$ and $y=-1$. This follows from the fact that $\lambda$ belongs to the continuous spectrum precisely when $\lvert\gamma(\lambda)\rvert\leq 1$, see \cite[Thm. ~5.3.2]{eastham1973}. It can also be shown that there exists a $\Lambda$ such that $\gamma(\lambda)>1$ for all $\lambda< \Lambda$, meaning that the first intersection of $\gamma$ with $y=1$ should give the start of the first closed interval. Additionally, it can be shown that the endpoint of the interval is at the first intersection of $\gamma$ with $y=-1$ (if it exists, otherwise the spectrum continues to $+\infty$), after which the second interval starts again at the next intersection of $\gamma$ with $y=-1$ and ends at the second intersection of $\gamma$ with $y=1$ and then repeats in a similar fashion (assuming that there are intersections so that the bands are not infinitely long) \cite[Thm.~2.3.1]{eastham1973}.

As we focus on the essential spectrum in a neighborhood of our embedded eigenvalue and we want to consider small values of $\lambda_0$, we only need to consider $m=0$ and $m=\pm 1$, as the spectra of the other ODE operators should not contain $\lambda_0$, if it is small enough. Indeed, for  $|m|=2$, the first band starts after $\lambda=2$, as can be computed numerically, and for other $m$, the first band starts even further away. We also note that the operators with $m=1$ and $m=-1$ have the same essential spectra.    

For $m=0$, as $\cos^2 z=(1+\cos 2z)/2$, the essential spectrum of this ordinary differential operator is the spectrum of the Schrödinger operator with a Mathieu potential, shifted by $1/2$ to the right. Hence, the first band of the essential spectrum for the ODE operator for $m=0$ starts at $\mu_0\approx 0.469$ and ends at a point which is approximately $1.242$. 

For $m=\pm 1$ the first band is starts around $\mu_m\approx 0.564$ and ends at approximately $0.572$  and the second one starts around $1.88$.

Going back to the operator $A_{per}$, it has essential spectrum which is the union of the essential spectra of the above ODE operators for $m\in\Z$. 

To construct the operator $A_0$, we will pick a $\lambda_0$ which belongs to the essential spectrum of the ODE operator associated with $m=0$, but not in the essential spectrum for the ODE operators associated with $m=\pm 1$. For example, we may choose $\lambda_0 = 0.5$. See Figure \ref{fig:spec} below for the spectral picture.
\begin{figure}[H]
\begin{center}
    \begin{tikzpicture}[domain=-1:2, samples=500,scale =30]

  \draw[->] (0.45,0) -- (0.6,0) node[right] {$\real\lambda$};
  \draw[red, very thick,-] (0.469,0) -- (0.6,0);
  \node at (0.69501,0) {$m=0$};

  \draw[->] (0.45,-1/30) -- (0.6,-1/30) node[right] {$\real\lambda$};
  \draw[red, very thick,-] (0.564,-1/30) -- (0.572,-1/30);
  \draw[-] (0.5-1/300,-1/30+1/300) -- (0.5+1/300,-1/30-1/300) node[below] {$\lambda_0$};
  \draw[-] (0.5-1/300,-1/30-1/300) -- (0.5+1/300,-1/30+1/300);
  \node at (0.7,-1/30) {$m=\pm 1$};

  \end{tikzpicture}\\~\
\end{center}
\caption{Spectra of the ODE operators corresponding to $m=0$ and $m=\pm 1$ respectively together with $\lambda_0$.}
\label{fig:spec}
\end{figure}
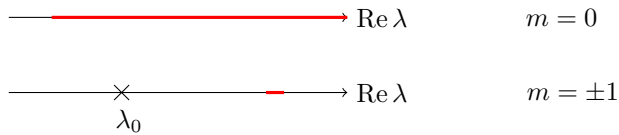

It follows from \cite[Thm.~1.1.2]{eastham1973} that there are linearly independent solutions $\psi_1$ and $\psi_2$ of the ODEs such that either $\psi_i(z)=e^{\alpha_iz}q_i(z)$, where $\alpha_i$ constants, not necessarily distinct and $q_i$ are periodic with the same period as the potential, or $\psi_1(z)=e^{\alpha z}q_1(z)$, $\psi_2(z)=e^{\alpha z}(zq_1(z)+q_2(z))$. Here, $\alpha_i$ are defined as $e^{\alpha_i p}=\rho_i$, where $\rho_i$ are the eigenvalues of the matrix
\[M=\begin{pmatrix}
    \phi_1(p) & \phi_2(p) \\
    \phi_1'(p) & \phi_2'(p)
\end{pmatrix},\]
where $\phi_1,\phi_2$ are the two linearly independent solutions from above, satisfying the initial conditions $\phi_1(0)=\phi_2'(0)=1$, $\phi_1'(0)=\phi_2(0)=0$. Note in particular that in this case $\det M=1$, and so $\rho_2=\frac{1}{\rho_1}$.
By letting $m=1$ and $\lambda_0=0.5$, we may now compute the $\alpha_i$ numerically, obtaining approximately $\pm 0.5835$. Now by denoting $\alpha=\lvert\alpha\rvert$ and defining 
\[f_{\pm}(z)=e^{\pm\alpha z}q_{\pm}(z),\]
where $q_+$ and $q_-$ are the periodic functions corresponding to the exponents $+\alpha$ and $-\alpha$, respectively. Since $f_{\pm}$ are two linearly independent solutions, they satisfy $f_{\pm}''(z)= \left((m+A_{\per}(z))^2-\lambda_0\right)f_{\pm}(z)$ with $f_{\pm}(z)$ decaying exponentially at $z\to\pm\infty$ respectively, we may define
\[v_*(z)=\dfrac{1}{2}(1-\tanh(\beta z))f_+(z)+ \dfrac{1}{2}(1+\tanh(\beta z))f_-(z), \quad \text{ for some } \beta >0.\]
Numerically, these functions can by found by solving for the eigenvectors of the matrix $M$ and then using them as initial conditions when solving the ODE.
Due to the symmetries of the equation, it should be noted that if we define
\[f_+(z)=f(z)=e^{\alpha z}q(z),\]
and $\real\alpha>0$, it follows that 
\[f_-(z)=f(-z)\]
is also a solution to the equation, with $f_+,f_- >0$ and linearly independent.

The solution in our case will look something like in Figure \ref{fig:eigenfct} below.

\begin{figure}[H]
    \centering
    \includegraphics[width=0.75\linewidth]{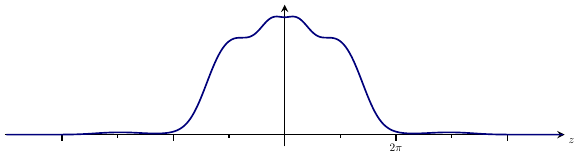}
    \caption{Eigenfunction of $\ell_m$ for $m=1$ corresponding to the eigenvalue $\lambda_0=0.5$.}
    \label{fig:eigenfct}
\end{figure}

Note that $v_*$ is bounded and decays exponentially to $0$ as $\lvert z\rvert\to 0$ and that, after some calculations, 
\begin{equation*}
    \dfrac{v_*''(z)}{v_*(z)}=\dfrac{2\beta\operatorname{sech}^2(\beta z)\left(\beta\tanh(\beta x)f_+(z)-f_+'(z)   -(\beta\tanh(\beta z)f_-(z) - f_-'(z))\right)}{(1-\tanh(\beta z))f_+(z)+(1+\tanh(\beta z))f_-(z)} +(m+A_{\per})^2 - \lambda_0.
\end{equation*}
Plugging in that $f_{\pm}'(z) = \pm\alpha e^{\pm\alpha z}q_{\pm}(z) + e^{\pm\alpha z}q_{\pm}'(z) = \pm\alpha f_{\pm}(z) + e^{\pm\alpha z}q_{\pm}'(z)$, we get
\footnotesize
\begin{equation*}
\begin{aligned}
    \dfrac{v_*''(z)}{v_*(z)}&= \dfrac{2\beta\operatorname{sech}^2(\beta z)\left(\left((\beta\tanh(\beta x)-\alpha)e^{\alpha z}q_{+}(z)  -(\beta\tanh(\beta z)-\alpha)e^{-\alpha z}q_{-}(z)\right) - (e^{\alpha z}q_{+}'(z) + e^{-\alpha z}q_{-}'(z)) \right)}{(1-\tanh(\beta z))e^{\alpha z}q_{+}(z)+(1+\tanh(\beta z))e^{-\alpha z}q_{-}(z)}\\
    &\quad +(m+A_{\per})^2 - \lambda_0. \\
\end{aligned}
\end{equation*}
\normalsize
Using the value $\beta=1$, we can write the localized part of the potential as 
\footnotesize
\begin{equation*}
    \begin{aligned}
        \dfrac{v_*''(z)}{v_*(z)}&-(m+A_{\per}(z))^2+\lambda_0 = S(z) \\
        &= \dfrac{2\operatorname{sech}^2(z)\left((f_+(z)+f_-(z) - ( \alpha e^{\alpha z}q_{+}(z) + e^{\alpha z}q_{+}'(z) - \alpha e^{-\alpha z}q_{-}(z) + e^{-\alpha z}q_{-}'(z)) -2v(z) \right)}{2v(z)}.
    \end{aligned}
\end{equation*}
\normalsize
But $f_{\pm}(z)=\pm \alpha e^{\pm\alpha z}q_{\pm}(z) + e^{\pm\alpha z}q_{\pm}'(z)$, and so we obtain
\begin{equation*}
    \begin{aligned}
        S(z) &= \dfrac{2\operatorname{sech}^2(z)\left((f_+(z)+f_-(z) -  f_+'(z) - f_-'(z) -2v(z) \right)}{2v_*(z)}. \\
    \end{aligned}
\end{equation*}
Note additionally that
\begin{equation*}
    \begin{aligned}
        \lvert S(z)\rvert \leq C e^{-(2-\alpha)\lvert z\rvert} \quad \text{ for all } z\in\R,
    \end{aligned}
\end{equation*}
and so clearly $S(z)\to 0$ as $\lvert z\rvert \to \infty$ exponentially as $\alpha <2$.

Then it follows that
\begin{equation*}
    \begin{aligned}
        (m+A_0(z))^2 = \dfrac{v_*''(z)}{v_*(z)}+\lambda_0 = (m+A_{\per}(z))^2+S(z),
    \end{aligned}
\end{equation*}
as illustrated in Figure \ref{fig:ptsq}.

\begin{figure}[H]
    \centering
    \includegraphics[width=0.75\linewidth]{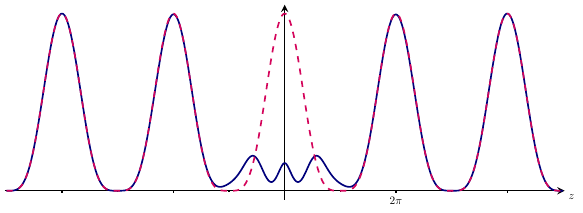}
    \caption{Potential of $\ell_m$ with the desired properties. The red dotted line is $\left(m+\cos(z)\right)^2$.}
    \label{fig:ptsq}
\end{figure}

Finally, solving for $A_0$, we get that

\begin{equation*}
    \begin{aligned}
        A_0(z) = -m+\sqrt{(m+A_{\per}(z))^2+S(z)},
    \end{aligned}
\end{equation*}
with $A_0$ plotted numerically in Figure \ref{fig:pt}.

\begin{figure}[H]
    \centering
    \includegraphics[width=0.75\linewidth]{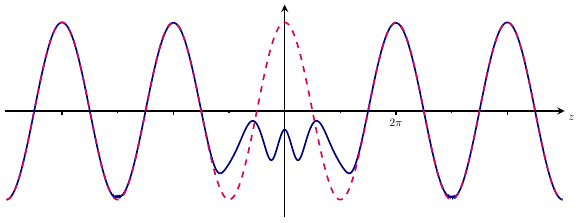}
    \caption{Magnetic potential of the operator with $\lambda_0=0.5$ as an embedded eigenvalue (blue line). The red dotted line is the periodic background potential $\cos(z)$.}
    \label{fig:pt}
\end{figure}

\section*{Acknowledgements}
J.J.~acknowledges the support of Lunds Universitet, where major parts of the research were carried out.

\printbibliography
\newpage

\end{document}